\documentclass[10pt,reqno]{amsart}
\usepackage{graphicx,epsfig}
\usepackage{epstopdf}
\usepackage{pdfpages}
\usepackage{amssymb}
\usepackage{empheq}
\usepackage{cases}
\usepackage{amsthm,amsmath}
\usepackage{caption,lipsum}
\usepackage{stmaryrd}
\usepackage{tabularx}
\usepackage{color}
\usepackage{empheq,float}
\DeclareGraphicsExtensions{.eps}
\usepackage{amssymb,amsmath,graphicx,amsfonts,euscript,mathrsfs}
\usepackage{hyperref}

\newtheorem{theorem}{Theorem}[section]
\newtheorem{lemma}[theorem]{Lemma}

\newtheorem{corollary}[theorem]{Corollary}
\theoremstyle{definition}

\theoremstyle{remark}
\newtheorem{remark}[theorem]{Remark}

\def\R{\mathbb{R}}
\def\Z{\mathbb{Z}}
\def\T{\mathbb{T}}

\def\P{\mathbb{P}}

\newcommand{\fe}{\mathrm{e}}

\newcommand{\bR}{{\mathbb R}}

\newcommand{\bT}{{\mathbb T}}
\newcommand{\bN}{{\mathbb N}}

\numberwithin{equation}{section}

\begin{document}

\title[Embedded low-regularity integrators for KdV equation]{Embedded exponential-type low-regularity integrators for KdV equation under rough data}

\author[Y. Wu]{Yifei Wu}
\address{\hspace*{-12pt}Y.~Wu: Center for Applied Mathematics, Tianjin University, 300072, Tianjin, China}
\email{yerfmath@gmail.com}

\author[X. Zhao]{Xiaofei Zhao}
\address{\hspace*{-12pt}X.~Zhao: School of Mathematics and Statistics \& Computational Sciences Hubei Key Laboratory, Wuhan University, Wuhan, 430072, China}
\email{matzhxf@whu.edu.cn}

%

\date{}

\dedicatory{}

\begin{abstract}
In this paper, we introduce a novel class of embedded exponential-type low-regularity integrators (ELRIs) for solving the KdV equation and establish their optimal convergence results under rough initial data. The schemes are explicit and efficient to implement. By rigorous error analysis, we first show that the ELRI scheme provides the first order accuracy in $H^\gamma$ for initial data in $H^{\gamma+1}$ for $\gamma>\frac12$. Moreover, by adding two more correction terms to the first order scheme, we show a second order ELRI that provides the second order accuracy in $H^\gamma$ for initial data in $H^{\gamma+3}$ for $\gamma\ge0$. The proposed ELRIs further reduce the regularity requirement of existing methods so far for optimal convergence. The theoretical results are confirmed by numerical experiments, and comparisons with existing methods illustrate the efficiency of the new methods.
\end{abstract}

\maketitle

\tableofcontents

\section{Introduction}
The Korteweg-de Vries (KdV)  equation is a classical model of paramount importance in mathematical studies and also for describing
the waves on shallow water surfaces.
 In this work, we are concerned with the numerical integration of the KdV equation on a torus:
\begin{equation}\label{model}
 \left\{\begin{split}
& \partial_tu(t,x)+\partial_x^3u(t,x)
 =\frac{1}{2}\partial_x(u(t,x))^2,
 \quad t>0,\ x\in\bT,\\
 &u(0,x)=u_0(x),\quad x\in\bT,
 \end{split}\right.
\end{equation}
where $\bT=(0,2\pi)$ is the torus, $u=u(t,x):\bR^{+}\times\bT\to\bR$ is the unknown and $u_0\in H^{s_0}(\bT)$ with some $0\leq s_0<\infty$ is the
given initial data.
Numerically, the KdV equation (\ref{model}) has already been extensively studied in the literature. Different kinds of numerical discretizations in space and time such as finite difference \cite{FD-kdv,FD1}, operator splitting \cite{splitting1,splitting2,splittingJCP,splitting0}, spectral methods \cite{Spectralkdv0,Spectralkdv1,Fourierkdv,Spectralkdv2} and discontinuous Galerkin method \cite{DG1,DG2} have been proposed and analyzed for solving the KdV equation by assuming that the solution of (\ref{model}) is smooth enough. When the solution of the equation is not sufficiently smooth in space, the numerical methods will fail to reach their optimal convergence rates in time and space, and then become less efficient. This is because that in the standard numerical discretizations, the truncation terms usually involve high order spatial derivatives of the solution, i.e. $\partial_x^ku(t,x)$ for some $k\geq0$.
As a matter of fact, the solution of the KdV equation in general could not be ideally smooth due to practical measurements or some randomness coming from the initial data or some potential \cite{SDE,kdv-wellposed}. In this work, we shall consider the KdV equation (\ref{model}) with some rough initial data.

Theoretically, the global well-posedness of the KdV equation for both the torus and whole space cases has already been established in $H^{-1}$ space \cite{CKSTT-03-KDV, KiVi-KdV, KaTo}, i.e. for any $u_0\in H^{s_0}(\T), s_0\ge -1$ and any positive time $T >0$, there exists a unique solution of \eqref{model} in a certain Banach space of functions
$X \subset C([0, T ];H^{s_0}(\T))$.
Such theoretical studies have also been carried out for the generalized KdV equations \cite{BaoWu,CKSTT-04-KDV}.
To address the accuracy of the numerical integrators for the KdV equation (\ref{model}) under rough initial data, many recent efforts \cite{FD-kdv,kdv-kath,splitting1,splitting2,kdv-wellposed,kdv-new1} have been made. Among them, the finite difference method has been analyzed in \cite{FD-kdv}.
The operator splitting methods were shown in \cite{splitting2}
that the first-order convergence rate
 $$\|u(t_n,\cdot)-u^n\|_{H^\gamma}\lesssim \tau,$$
 and the second-order convergence rate
 $$\|u(t_n,\cdot)-u^n\|_{H^\gamma}\lesssim \tau^2,$$
up to some finite time are achieved  respectively for solutions in $H^{\gamma+3}$ and $H^{\gamma+5}$ spaces for some $\gamma>0$, where $\tau>0$ denotes the time step and $u^n$ denotes the numerical solution at $t_n=n\tau$. To bring down the regularity requirements, there is a
 recent trend to  propose the so-called \emph{low-regularity integrators (LRIs)}. This has already been considered for some important dispersive models.
For example, for the cubic nonlinear Schr\"{o}dinger equation in one dimension, the first order and second order convergence rates in $H^\gamma$ have been achieved under respectively only
$H^{\gamma+1}$-data \cite{lownls3,lownls} and $H^{\gamma+2}$-data \cite{lownls2}, and for the one-dimensional quadratic nonlinear Schr\"{o}dinger equation \cite{lownls2,lownls3} or
the nonlinear Dirac equations \cite{diraclow}, only $H^{\gamma}$-data is needed. To design the LRIs for the KdV equation, there are two main difficulties. The first one which is a common difficulty of designing LRIs, is to  drop as less spatial derivatives in the approximation as possible but meanwhile keep the scheme being defined point-wisely in the physical space rather than in the Fourier space. This is very important to guarantee the efficiency of the scheme \cite{lownls2,lownls,WuZhao-1}. The second one is the presence of the Burgers-nonlinearity in the KdV equation (\ref{model}) that naturally involves spatial derivative and makes the task harder compared with others such as the cubic Schr\"{o}dinger equation.
To breaking the ground, by introducing the twisted variable $v:=\fe^{\partial_x^3t}u$ and the Duhamel's formula at $t_n$:
\begin{equation}\label{v eq}
  v(t_n+\tau,x)=v(t_n,x)+\frac{1}{2}\int_0^\tau\fe^{(t_n+s)\partial_x^3}
  \partial_x\left(\fe^{-(t_n+s)\partial_x^3}v(t_n+s,x)\right)^2ds,
\end{equation}
 \cite{kdv-kath} proposed an exponential-type numerical scheme by letting $v(t_n+s,x)\approx v(t_n,x)$, and
 the integration for $s$ was found exactly and explicitly in the physical space. This strategy gives rise to a
 first order LRI scheme  in $H^1$
for initial value $u_0\in H^{3}$ as proved in \cite{kdv-kath}. Moreover, by approximation
$v(t_n+s,x)\approx v(t_n,x)+s\partial_tv(t_n,x)$
as also been outlined in \cite{kdv-kath} and later proved rigorously in our recent work \cite{WuZhao-1} that
the second-order convergence is reached in $H^\gamma$ for the initial data from $H^{\gamma+4}$ for any $\gamma\geq0$. That is to say for integrating the KdV equation (\ref{model}), in order to reach the optimal first order accuracy and second order accuracy respectively, the
minimum regularity requirements among all the possible numerical schemes so far are the boundedness of two and four
additional spatial derivatives of the solution. In the recent investigation \cite{Schratznew}, such regularity requirements have been considered as essential for all direct integration approaches by the resonance structure.

To further bring down the regularity requirement and make it more accurate and efficient to solve the KdV equation under rough data, in this work, we are going to introduce two novel \emph{embedded low-regularity integrators (ELRIs)} of the exponential-type \cite{Hochbruck}
for solving (\ref{model}) by incorporating more strategies from the harmonic analysis. The main idea is to work on an embedded integral form of (\ref{model}) by iterating the Duhamel's formula (\ref{v eq}) once instead of directly working on (\ref{v eq}). Note that the iterative strategy here is not aiming for high order numerical approximations but for saving the regularity of the solution, where the nested integral in the embedded form after integration-by-parts does help. We then use the multiplier analysis to construct the explicit numerical scheme in the physical space. This idea was inspired by the work in \cite{CKSTT-03-KDV}, and it can also be represented by the normal form in \cite{BaoWu}.  As an example to illustrate the great role that the multiplier analysis plays in the approximation process, we consider here the main obstacle term in Fourier space occurred after the integration-by-parts:
$$
\int_0^\tau\!\!\sum\limits_{\xi=\xi_1+\xi_2+\xi_3  \atop  \xi_1,\xi_2,\xi_3\ne 0} \frac{1}{\xi_1}\fe^{-i(t_n+s)(\xi^3-\xi_1^3-\xi_2^3-\xi_3^3)}
\>\widehat{v}(t_n,\xi_1)\widehat{v}(t_n,\xi_2)\widehat{v}(t_n,\xi_3)\,ds.
$$
Here $\widehat{v}$ denotes the Fourier transform of $v$.
The phase function $\xi^3-\xi_1^3-\xi_2^3-\xi_3^3$ can not be integrated in the physical space exactly, which prevents us from defining an explicit point-wise scheme in the physical space. To overcome this difficulty, the key observation is to replace the multiplier $\frac1{\xi_1}$ by $[\frac1{\xi_1}]_{sym}$, where  $[\cdot]_{sym}$ is the symmetrization of the multiplier, and  note
\begin{align*}
\Big[\frac{1}{\xi_1}\Big]_{sym}=&\frac13\Big(\frac{1}{\xi_1}+\frac{1}{\xi_2}+\frac{1}{\xi_3}\Big)
=\frac{\xi^3-\xi_1^3-\xi_2^3-\xi_3^3}{9\xi\xi_1\xi_2\xi_3} +\frac1{3\xi}.
\end{align*}
This allows us to
split the above integral term into two parts:
$$
\frac13\int_0^\tau\!\!\sum\limits_{\xi=\xi_1+\xi_2+\xi_3  \atop  \xi_1,\xi_2,\xi_3\ne 0}\frac{\xi^3-\xi_1^3-\xi_2^3-\xi_3^3}{3\xi\xi_1\xi_2\xi_3}\fe^{-i(t_n+s)(\xi^3-\xi_1^3-\xi_2^3-\xi_3^3)}
\>\widehat{v}(t_n,\xi_1)\widehat{v}(t_n,\xi_2)\widehat{v}(t_n,\xi_3)\,ds
$$
and
$$
\frac1{3\xi}\int_0^\tau\!\!\sum\limits_{\xi=\xi_1+\xi_2+\xi_3  \atop  \xi_1,\xi_2,\xi_3\ne 0}\fe^{-i(t_n+s)(\xi^3-\xi_1^3-\xi_2^3-\xi_3^3)}
\>\widehat{v}(t_n,\xi_1)\widehat{v}(t_n,\xi_2)\widehat{v}(t_n,\xi_3)\,ds,
$$
where the first part can be exactly integrated in the physical space. As for the second part, for a first order scheme, we simply approximate it by:
$$
\frac\tau{3\xi}\sum\limits_{\xi=\xi_1+\xi_2+\xi_3  \atop  \xi_1,\xi_2,\xi_3\ne 0}
\fe^{-it_n(\xi^3-\xi_1^3-\xi_2^3-\xi_3^3)}
\>\widehat{v}(t_n,\xi_1)\widehat{v}(t_n,\xi_2)\widehat{v}(t_n,\xi_3),
$$
which only loses one spatial derivative, thanks to the factor $\frac{1}{\xi}$ in front and the fact  $\xi^3-\xi_1^3-\xi_2^3-\xi_3^3
=3(\xi\xi_1\xi_2+\xi\xi_1\xi_3+\xi\xi_2\xi_3-\xi_1\xi_2\xi_3)$. A more accurate approximation is used to the second part in order to get the second order accuracy.

The proposed final schemes are of exponential-type and are fully explicit with point-wise definition in the physical space. As will be shown by our rigorous theoretical estimates,  the ELRIs  are able to get the desired
\textbf{first order} and the \textbf{second order} accuracy with respectively only \textbf{one} and \textbf{three} additional bounded spatial derivatives.
The main convergence results will be presented in \textbf{Theorem \ref{thm:convergence}} and \textbf{Theorem \ref{thm:convergence-2ord}}. Numerical experiments will be given in the end to justify the theoretical estimates and to compare with the performance of the existing LRIs from \cite{kdv-kath}. As will be shown by the numerical results, the proposed ELRIs are indeed much more accurate and efficient than the LRIs.

The rest of the paper is organized as follows. In Section \ref{sec:scheme}, we derive the schemes of the first order and second order ELRI schemes and state the main convergence results in Theorem \ref{thm:convergence} and Theorem \ref{thm:convergence-2ord}. In Section \ref{sec:1ord-proof} and Section \ref{sec:2ord-proof},
we present the rigorous proofs of the first-order convergence result and  the second-order convergence result in a sequel.
Numerical confirmations are reported in Section \ref{sec:numerical} and conclusions are drawn in Section \ref{sec:conclusion}.

\section{Embedded low-regularity integrators: schemes and derivation}\label{sec:scheme}
In this section, we will construct our new embedded low-regularity exponential-type integrators. To this purpose, we shall first present some notations and tools for the convenience of the derivation and numerical analysis later. Some of them are employed from \cite{CKSTT-03-KDV}.

\subsection{Some notations and facts}
Denote $\langle \cdot\rangle=(1+|\cdot|^2)^\frac12$  and  $|\nabla|^{s_0}=(-\partial_{xx})^{\frac{{s_0}}{2}}$ for ${s_0}>0$.
We define
$(d\xi)$ to be  the normalized counting measure on
$\Z$ such that
\begin{equation*}
\displaystyle\int
a(\xi)\,(d\xi)
=
\sum\limits_{\xi\in
\Z} a(\xi).
\end{equation*}
The Fourier transform of a function $f(x)$ on $\T$ is defined by
$$
\mathcal{F}(f)(\xi)=\widehat{f}(\xi)=\frac1{2\pi}\displaystyle\int_{\bT}
\fe^{- i   x\xi}f( x)\,d x,
$$
and thus the Fourier inversion formula reads
$$
f(x)=\displaystyle\int \fe^{i  x\xi} \widehat{f}(\xi)\,(d\xi).
$$
Then the following usual properties of the Fourier transform hold:
\begin{eqnarray*}
 &\|f\|_{L^2(\bT)}
 = \sqrt{2\pi}\big\|\widehat{f}\big\|_{L^2((d \xi))} \quad \mbox{(Plancherel)}; \\
 &\displaystyle\langle f,g\rangle=\int_\bT f(x)\overline{g(x)}\,dx
 =2\pi \displaystyle\int \widehat{f}(\xi)\overline{\widehat{g}(\xi)}\,(d\xi)\quad \mbox{(Parseval)} ; \\
 & \widehat{(fg)}(\xi)=\displaystyle\int
  \widehat{f}(\xi-\xi_1)\widehat{g}(\xi_1) \,(d\xi_1) \quad \mbox{(Convolution)}.
\end{eqnarray*}

The Sobolev space $H^{s_0}(\bT)$ for ${s_0}\geq0$ has the equivalent norm,
$$
\big\|f\big\|_{H^{s_0}(\bT)}=
\big\|J^{s_0} f\big\|_{L^2(\bT)}=\sqrt{2\pi}\left\|\langle\xi\rangle^{s_0}
\widehat{f}(\xi)\right\|_{L^2((d\xi))},
$$
where we denote the operator
$$J^{s_0}=(1-\partial_{xx})^\frac {s_0}2.$$
Moreover, we define $\partial_x^{-1}$ for function $f(x)$ on $\bT$ as
\begin{equation}\label{def:px-1}
\widehat{(\partial_x^{-1}f)}(\xi)
=\Bigg\{ \aligned
    &(i\xi)^{-1}\widehat f(\xi),\quad &\mbox{when } \xi\ne 0,\\
    &0,\quad &\mbox{when } \xi= 0.
   \endaligned
\end{equation}
We denote $\P$ the orthogonal projection onto mean zero functions, that is
\begin{equation}\label{P def}
\P f(x)=f(x)-\frac1{2\pi}\int_\T f(x)\,dx.
\end{equation}
In particular, it is worth noting that from the definition of $\partial_x^{-1}$ in \eqref{def:px-1}, we have the special `formula':
\begin{align}
\partial_x^{-1} \partial_x f(x)=\P f(x).\label{par-x+x}
\end{align}

To analyze the resonance of the Fourier frequencies, for simplicity we denote
$$
\alpha_{k+1}=\xi^3-\xi_1^3-\cdots-\xi_k^3,\quad k\ge 2.
$$
Then if $\xi=\xi_1+\cdots+\xi_k$, we have that
\begin{subequations}
\begin{align}
\alpha_3&=3\xi\xi_1\xi_2; \label{formu-a3}\\
\alpha_4&=
3(\xi\xi_1\xi_2+\xi\xi_1\xi_3+\xi\xi_2\xi_3-\xi_1\xi_2\xi_3).\label{formu-a4}
\end{align}
\end{subequations}

For convenience of the derivation of schemes,  we shall assume that the zero-mode/average of the initial value of (\ref{model}) is zero, that is $\widehat{ u_0}(0)=0$. Otherwise, we may consider instead
\begin{equation}
\tilde u(t,x):= u\left(t,x-\widehat{u_0}(0) t\right)-\widehat{u_0}(0),\quad t\geq0,\ x\in\bT,\label{shift}
\end{equation}
and one may note that $\tilde u$ also obeys the same KdV equation of \eqref{model} with initial data $\tilde u_0:=u_0-\widehat u_0(0)$.
For numerical discretizations, we denote $\tau>0$ as the time step and $t_n=n\tau$ for $n\in\bN$ as the grid points in time. When it is not misleading, we will omit the spatial variable $x$ of a function of space-time $f(t,x)$ for simplicity, e.g. $f(t)=f(t,x)$.

Now by direct calculations, we give the first key fact for our derivation of schemes.
\begin{lemma}(Integration-by-parts)\label{lem:1-form}
\begin{itemize}
  \item[(i)]
  Let  the space-time functions $f(t,x),g(t,x)\in L^2(\bT)$ with zero-average  $\widehat{f}(t,0)=\widehat{g}(t,0)=0$ for $t\in[0,\tau]$, then we have for any $t_n\geq0$,
  \begin{align*}
   &\int_0^\tau \fe^{(t_n+t)\partial_x^3}\partial_x\left(\fe^{-(t_n+t)\partial_x^3}f(t)\cdot \fe^{-(t_n+t)\partial_x^3}g(t)\right)\,dt\\
  =&
  \frac13 \fe^{t_{n+1}\partial_x^3}\left(\fe^{-t_{n+1}\partial_x^3}\partial_x^{-1}f(\tau)\cdot \fe^{-t_{n+1}\partial_x^3}\partial_x^{-1}g(\tau)\right)
  -\frac13\fe^{t_{n}\partial_x^3}\left(\fe^{-t_{n}\partial_x^3}\partial_x^{-1}f(0)\cdot \fe^{-t_{n}\partial_x^3}\partial_x^{-1}g(0)\right)\\
  &-\frac13 \int_0^\tau \fe^{(t_{n}+t)\partial_x^3}\Big[\fe^{-(t_{n}+t)\partial_x^3}\partial_x^{-1}\partial_tf(t)\cdot \fe^{-(t_{n}+t)\partial_x^3}\partial_x^{-1}g(t)\\
  &\qquad\qquad\qquad\qquad+\fe^{-(t_{n}+t)\partial_x^3}\partial_x^{-1}f(t)\cdot \fe^{-(t_{n}+t)\partial_x^3}\partial_x^{-1}\partial_tg(t)\Big]\,dt.
  \end{align*}
  \item[(ii)]
  Let the space functions $f_j(x)\in L^2(\bT)$ (time-independent) with $\widehat{f_j}(0)=0$ for $j=1,2,3$, then we have for any $t_n\geq0$,
  \begin{align*}
   &\int_0^\tau\!\! \int_\T \fe^{(t_n+t)\partial_x^3}\left(\fe^{-(t_n+t)\partial_x^3}f_1\cdot
   \fe^{-(t_n+t)\partial_x^3}f_2\cdot\fe^{-(t_n+t)\partial_x^3}f_3\right)\,dxdt\\
  =&
-  \frac13\int_\T \left(\fe^{-t_{n+1}\partial_x^3}\partial_x^{-1}f_1\cdot
\fe^{-t_{n+1}\partial_x^3}\partial_x^{-1}f_2\cdot
\fe^{-t_{n+1}\partial_x^3}\partial_x^{-1}f_3\right)\,dx\\
 &+\frac13\int_\T\left(\fe^{-t_{n}\partial_x^3}\partial_x^{-1}f_1\cdot
 \fe^{-t_{n}\partial_x^3}\partial_x^{-1}f_2\cdot
 \fe^{-t_{n}\partial_x^3}\partial_x^{-1}f_3\right)\,dx.
  \end{align*}
  \end{itemize}
\end{lemma}

\begin{proof}
(i) By taking the Fourier transform, we get for any $t_n\geq0$,
  \begin{align*}
&\mathcal F\left(\int_0^\tau \fe^{(t_n+t)\partial_x^3}\partial_x
\left(\fe^{-(t_n+t)\partial_x^3}f(t)\cdot \fe^{-(t_n+t)\partial_x^3}g(t)\right)\,dt\right)(\xi)\\
  =&i\xi\int_0^\tau\!\int_{\xi=\xi_1+\xi_2}\! \fe^{-i(t_n+t)\alpha_3}\widehat{f}(t,\xi_1)\widehat{g}(t,\xi_2)\,(d\xi_1)dt,
  \end{align*}
  where $\alpha_3$ is given in (\ref{formu-a3}).
  Note that the right-hand side vanishes when $\xi=0$, so we only need to consider the case $\xi\ne 0$. Then from the formula
  $$
  \fe^{-i(t_n+t)\alpha_3}
  =-\frac{1}{3i\xi\xi_1\xi_2}\frac{d}{dt}\Big(\fe^{-i(t_n+t)\alpha_3}\Big),
  $$
   and integration-by-parts for $t$, we find
  \begin{align*}
&\mathcal F\Big(\int_0^\tau \fe^{(t_n+t)\partial_x^3}\partial_x\big(\fe^{-(t_n+t)\partial_x^3}f(t)\cdot \fe^{-(t_n+t)\partial_x^3}g(t)\big)\,dt\Big)(\xi)\\
=&-\int_{\xi=\xi_1+\xi_2}\! \frac{1}{3\xi_1\xi_2}\fe^{-it_{n+1}\alpha_3}
  \widehat{f}(\tau,\xi_1)\widehat{g}(\tau,\xi_2)\,(d\xi_1)
  +\int_{\xi=\xi_1+\xi_2}\! \frac{1}{3\xi_1\xi_2}\fe^{-it_{n}\alpha_3}
 \widehat{f}(0,\xi_1)\widehat{g}(0,\xi_2)\,(d\xi_1)\\
&\,\,\, +\int_0^\tau\!\int_{\xi=\xi_1+\xi_2}\! \frac{1}{3\xi_1\xi_2}\fe^{-i(t_{n}+t)\alpha_3}
 \Big(\partial_t\widehat{f}(t,\xi_1)\widehat{g}(t,\xi_2)+\widehat{f}(t,\xi_1)
 \partial_t\widehat{g}(t,\xi_2)\Big)\,(d\xi_1) dt.
  \end{align*}
 Note that the right-hand side of the last equality above vanishes when $\xi=0$.
Then this gives the claimed equality by the inverse Fourier transform.

(ii) Note that for any $t\in\R$ and $f(x)$ on $\bT$,
$$
\int_\T \fe^{it\partial_x^3} f(x)\,dx=\int_\T f(x)\,dx.
$$
Then we have
 \begin{align*}
   \int_0^\tau\!\! \int_\T &\fe^{(t_n+t)\partial_x^3}\left(\fe^{-(t_n+t)\partial_x^3}f_1\cdots \fe^{-(t_n+t)\partial_x^3}f_3\right)\,dxdt
  =
\int_0^\tau\!\! \int_\T \left(\fe^{-(t_n+t)\partial_x^3}f_1\cdots \fe^{-(t_n+t)\partial_x^3}f_3\right)\,dxdt.
  \end{align*}
Then similarly as above,  by using Parseval's identity, we get that for any $t_n\geq0$,
 \begin{align*}
&\quad\int_0^\tau \!\! \int_\T \fe^{(t_n+t)\partial_x^3}\left(\fe^{-(t_n+t)\partial_x^3}f_1\cdots \fe^{-(t_n+t)\partial_x^3}f_3\right)\,dxdt\\
  =&2\pi\int_0^\tau\!\!\int_{\xi_1+\xi_2+\xi_3=0}\! \fe^{i(t_n+t)(\xi_1^3+\xi_2^3+\xi_3^3)}\widehat{f}(\xi_1)
  \widehat{f}(\xi_2)\widehat{f}(\xi_3)\,(d\xi_1)(d\xi_2)dt.
  \end{align*}
 Note by (\ref{formu-a4}) that for $\xi_1+\xi_2+\xi_3=0$,
 $$
 \xi_1^3+\xi_2^3+\xi_3^3=3\xi_1\xi_2\xi_3.
 $$
 Then  we have the formula
  $$
 \int_0^\tau \fe^{i(t_n+t)(\xi_1^3+\xi_2^3+\xi_3^3)}\,dt
  =\frac{1}{3i\xi_1\xi_2\xi_3}\left(\fe^{it_{n+1}(\xi_1^3+\xi_2^3+\xi_3^3)}-\fe^{it_{n}(\xi_1^3+\xi_2^3+\xi_3^3)}\right),
  $$
and thus
  \begin{align*}
&\int_0^\tau\!\!\int_{\xi_1+\xi_2+\xi_3=0}\! \fe^{i(t_n+t)(\xi_1^3+\xi_2^3+\xi_3^3)}
\widehat{f}(\xi_1)\widehat{f}(\xi_2)\widehat{f}(\xi_3)\,(d\xi_1)(d\xi_2)dt\\
=&\int_{\xi_1+\xi_2+\xi_3=0}\! \frac{1}{3i\xi_1\xi_2\xi_3}\fe^{it_{n+1}(\xi_1^3+\xi_2^3+\xi_3^3)}
\widehat{f}(\xi_1)\widehat{f}(\xi_2)\widehat{f}(\xi_3)\,(d\xi_1)(d\xi_2)\\
&\,\,\, -
\int_{\xi_1+\xi_2+\xi_3=0}\! \frac{1}{3i\xi_1\xi_2\xi_3}\fe^{it_{n}(\xi_1^3+\xi_2^3+\xi_3^3)}
\widehat{f}(\xi_1)\widehat{f}(\xi_2)\widehat{f}(\xi_3)\,(d\xi_1)(d\xi_2).
  \end{align*}
This gives the claimed equality by Parseval's identity again.
\end{proof}

\subsection{First order scheme} With the above preparation, we now start to derive the first order low-regularity integrator for (\ref{model}).
By introducing the twisted variable \cite{kdv-kath}
\begin{equation}\label{twist}v(t,x):=\fe^{t\partial_x^3}u(t,x),\quad t\geq0,\ x\in\bT,\end{equation}
the KdV equation (\ref{model}) becomes:
\begin{equation}\label{vt eq}
  \partial_t v(t,x)=\frac{1}{2}\fe^{t\partial_x^3}
  \partial_x\left(\fe^{-t\partial_x^3}v(t,x)\right)^2,\quad
  t\geq0,\ x\in\bT.
\end{equation}
By the Duhamel formula \eqref{v eq}, we have the mild solution for some $n\geq0$,
\begin{equation*}
  v(t_{n}+\tau)=v(t_n)+\frac{1}{2}\int_0^\tau\fe^{(t_n+s)\partial_x^3}
  \partial_x\left(\fe^{-(t_n+s)\partial_x^3}v(t_n+s)\right)^2ds,\quad n\geq0.
\end{equation*}
Instead of directly approximating $v(t_n+s)$ in the above integrand as was done in the previous work \cite{kdv-kath}, we iterate the Duhamel's formula for one time in order to incorporate more inherent structure of the equation.
That is to insert the Duhamel's formula of $v(t_n+s)$ into the above equation, and we get
\begin{align}
&v(t_{n+1})\nonumber\\
=&v(t_n)+\frac{1}{2}\int_0^\tau\fe^{(t_n+s)\partial_x^3}
\partial_x\left[\fe^{-(t_n+s)\partial_x^3}\left(v(t_n)+\frac12\int_0^s \fe^{(t_n+\rho)\partial_x^3}
  \partial_x\left(\fe^{-(t_n+\rho)\partial_x^3}v(t_n+\rho)\right)^2\,d\rho\right)
  \right]^2ds.\label{nest v}
\end{align}
We emphasize that this iterative strategy applied here is not aiming for high order accurate approximations as used in the literature (see \cite{kdv-kath,WuZhao-1} and some other works in general \cite{Schratznew,diraclow,lownls2}), but for saving the spatial derivatives, where the nested integral will help to maintain some regularity in the approximation process.

We now begin our numerical approximation based on the above formula.  First of all,  let $v(t_n+\rho)\approx v(t_n)$ in the above integrand and we get
\begin{align}
 &v(t_{n+1})\nonumber\\
 \approx &v(t_n)+\frac{1}{2}\int_0^\tau\fe^{(t_n+s)\partial_x^3}
\partial_x\left[\fe^{-(t_n+s)\partial_x^3}\left(v(t_n)+\frac12\int_0^s \fe^{(t_n+\rho)\partial_x^3}
  \partial_x\left(\fe^{-(t_n+\rho)\partial_x^3}v(t_n)\right)^2\,
  d\rho\right)\right]^2ds.\label{v app1}
\end{align}
By further dropping  the $O(\tau^3)$-term in (\ref{v app1}), we get
\begin{equation}\label{v app}
  v(t_{n+1})\approx v(t_n)+\frac{1}{2}(I_1(t_n)+I_2(t_n)),\quad n\geq0,
\end{equation}
where we denote
\begin{align*}
 &I_1(t_n):=\int_0^\tau\fe^{(t_n+s)\partial_x^3}
 \partial_x\left(\fe^{-(t_n+s)\partial_x^3}v(t_n)\right)^2ds,\\
 & I_2(t_n):=\int_0^\tau\!\!\int_0^s \!\!\fe^{(t_n+s)\partial_x^3}
 \partial_x\left[\fe^{-(t_n+s)\partial_x^3}v(t_n)\cdot
 \fe^{(\rho-s)\partial_x^3}
  \partial_x\left(\fe^{-(t_n+\rho)\partial_x^3}v(t_n)\right)^2\,d\rho\Big)\right]ds.
\end{align*}
We will  show later rigorously that the dropped $O(\tau^3)$-term in (\ref{v app1}) can be treated as a $O(\tau^2)$ truncation with only one loss of derivative.
As was found in the pioneering work \cite{kdv-kath} or as a special case of our Lemma \ref{lem:1-form} (i), the integrand in $I_1(t_n)$  can be exactly integrated in the physical space:
\begin{equation}
I_1(t_n)=\frac13\fe^{t_{n+1}\partial_x^3}\left(\fe^{-t_{n+1}\partial_x^3}
\partial_x^{-1}v(t_n)\right)^2
-\frac13\fe^{t_n\partial_x^3}\left(\fe^{-t_n\partial_x^3}\partial_x^{-1}v(t_n)\right)^2,
\quad n\geq0.
\label{est:15.15-1}
\end{equation}
This is very important not only for the accuracy but also for the efficiency of the scheme, since the point-wise formula in the physical space can be efficiently implemented by Fourier spectral method.
However, this is \emph{not} possible for $I_2(t_n)$ unfortunately. Therefore, the following effort is made to approximate $I_2(t_n)$ in order to evaluate it in the physical space, where we pay particular attention to only afford one loss of derivative. The embedded temporal integral does provide flexibility to make this possible as we shall see.

We define for some function $w$ on $\bT$ and $s\geq0$,
\begin{align}\label{def-Fn}
F_n(w,s):=&\int_0^s \fe^{(t_n+\rho)\partial_x^3}\partial_x\left(\fe^{-(t_n+\rho)\partial_x^3}
w\right)^2\,d\rho\nonumber\\
=&\frac13 \fe^{(t_{n}+s)\partial_x^3}\left(\fe^{-(t_{n}+s)\partial_x^3}\partial_x^{-1}w\right)^2
  -\frac13\fe^{t_{n}\partial_x^3}\left(\fe^{-t_{n}\partial_x^3}\partial_x^{-1}w\right)^2,
  \quad n\geq0,
\end{align}
which is obtained similarly by the fact in (\ref{est:15.15-1}),
and then $I_2(t_n)$ reads
$$I_2(t_n)=\int_0^\tau\fe^{(t_n+s)\partial_x^3}
 \partial_x\left(\fe^{-(t_n+s)\partial_x^3}v(t_n)\cdot \fe^{-(t_n+s)\partial_x^3}F_n(v(t_n),s)\right)\,ds.$$
 Noting that $\widehat{F_n(v,s)}(0)=0$, so by applying the integration-by-parts in $I_2(t_n)$ with the result of Lemma \ref{lem:1-form} (i), we find
\begin{subequations}\label{10.40}
\begin{align}
I_2(t_n)
=&\frac13\fe^{t_{n+1}\partial_x^3}\Big(\fe^{-t_{n+1}\partial_x^3}\partial_x^{-1}v(t_n)\cdot \fe^{-t_{n+1}\partial_x^3}\partial_x^{-1}F_n\big(v(t_n),\tau\big)\Big)\label{10.40-1}\\
  &-\frac13\fe^{t_{n}\partial_x^3}\Big(\fe^{-t_{n}\partial_x^3}\partial_x^{-1}v(t_n)\cdot \fe^{-t_{n}\partial_x^3}\partial_x^{-1}F_n\big(v(t_n),0\big)\Big)\label{10.40-2}\\
 &-\frac13 \int_0^\tau \fe^{(t_{n}+s)\partial_x^3}\Big(
\fe^{-(t_{n}+s)\partial_x^3}\partial_x^{-1}v(t_n)\cdot \fe^{-(t_{n}+s)\partial_x^3}\partial_x^{-1}\partial_sF_n\big(v(t_n),s\big)\Big)\,ds.\label{10.40-3}
\end{align}
\end{subequations}
By (\ref{def-Fn}), we find
\begin{equation}\begin{split}
\eqref{10.40-1}
=&\frac1{9}\fe^{t_{n+1}\partial_x^3}\Big(\fe^{-t_{n+1}\partial_x^3}\partial_x^{-1}v(t_n)
\cdot\partial_x^{-1}\big(\fe^{-t_{n+1}\partial_x^3}\partial_x^{-1}v(t_n)\big)^2\Big)\\
&-\frac1{9}\fe^{t_{n+1}\partial_x^3}\Big(\fe^{-t_{n+1}\partial_x^3}\partial_x^{-1}v(t_n)
\cdot\fe^{-\tau\partial_x^3}\partial_x^{-1}\big(\fe^{-t_{n}\partial_x^3}
\partial_x^{-1}v(t_n)\big)^2\Big),
\end{split}\label{est:10.40-1}
\end{equation}
and since $F_n\big(v,0\big)=0$,  we have
\begin{align}
\eqref{10.40-2}
=0.\label{est:10.40-2}
\end{align}

Now it counts down to handle the term \eqref{10.40-3} in $I_2(t_n)$. Firstly, note that
\begin{align*}
\partial_sF_n(v,s)=\fe^{(t_n+s)\partial_x^3}\partial_x\big(\fe^{-(t_n+s)\partial_x^3}v\big)^2,
\end{align*}
so by \eqref{par-x+x},
\begin{align*}
\partial_x^{-1}\partial_sF_n(v,s)=&\P\Big[\fe^{(t_n+s)\partial_x^3}\big(\fe^{-(t_n+s)\partial_x^3}v\big)^2\Big]\\
=&\fe^{(t_n+s)\partial_x^3}\big(\fe^{-(t_n+s)\partial_x^3}v\big)^2-\frac1{2\pi}\int_\T \fe^{(t_n+s)\partial_x^3}\big(\fe^{-(t_n+s)\partial_x^3}v\big)^2\,dx\\
=&\fe^{(t_n+s)\partial_x^3}\big(\fe^{-(t_n+s)\partial_x^3}v\big)^2-\frac1{2\pi}\int_\T v^2\,dx,
\end{align*}
with the operator $\P$ defined in (\ref{P def}).
Hence, we find
\begin{align}
&\eqref{10.40-3}-
\frac\tau{6\pi} \partial_x^{-1}v(t_n)\int_\T \left(v(t_n)\right)^2\,dx\nonumber\\
=&
-\frac13 \int_0^\tau \fe^{(t_{n}+s)\partial_x^3}\Big(
\fe^{-(t_{n}+s)\partial_x^3}\partial_x^{-1}v(t_n)\cdot \big(\fe^{-(t_n+s)\partial_x^3}v(t_n)\big)^2\Big)\,ds\label{10.40-3-2}.
\end{align}
Next, we compute the zero Fourier frequency of \eqref{10.40-3-2} and the others separately.
On one hand, by the formula in Lemma \ref{lem:1-form} (ii),  the zero frequency of \eqref{10.40-3-2} reads
\begin{align*}
\mathcal F\eqref{10.40-3-2}(0)
=&-\frac1{6\pi} \int_0^\tau \!\!\int_\T\fe^{(t_{n}+s)\partial_x^3}\Big(
\fe^{-(t_{n}+s)\partial_x^3}\partial_x^{-1}v(t_n)\cdot \big(\fe^{-(t_n+s)\partial_x^3}v(t_n)\big)^2\Big)\,dxds\notag\\
=&\frac1{18\pi}\int_\T\fe^{-t_{n+1}\partial_x^3}\partial_x^{-2}v(t_n)
\cdot\big(\fe^{-t_{n+1}\partial_x^3}\partial_x^{-1}v(t_n)\big)^2\,dx\notag\\
&-\frac1{18\pi}\int_\T\fe^{-t_{n}\partial_x^3}\partial_x^{-2}v(t_n)
\cdot\big(\fe^{-t_{n}\partial_x^3}\partial_x^{-1}v(t_n)\big)^2\,dx.
\end{align*}
By directly checking in the Fourier space, the right-hand side of the above is found simply as
\begin{align}
\mathcal F\eqref{10.40-3-2}(0)
=&-\mathcal F\eqref{est:10.40-1}(0).\label{zero-mod}
\end{align}
On the other hand, for $\xi\ne 0$ in the Fourier space of \eqref{10.40-3-2},  we have
\begin{align*}
\mathcal F\eqref{10.40-3-2}(\xi)
 =&\frac i3\int_0^\tau\!\!\int_{\xi=\xi_1+\xi_2+\xi_3\neq0} \frac{1}{\xi_1}\fe^{-i(t_n+s)\alpha_4}
 \widehat{v}(t_n,\xi_1)\widehat{v}(t_n,\xi_2)\widehat{v}(t_n,\xi_3)\,d(\xi_1)d(\xi_2)ds.
\end{align*}
By symmetry, the above is equal to
\begin{align*}
\frac i{9}\int_0^\tau\!\!\int_{\xi=\xi_1+\xi_2+\xi_3\neq0} \Big(\frac{1}{\xi_1}+\frac{1}{\xi_2}+\frac{1}{\xi_3}\Big) \fe^{-i(t_n+s)\alpha_4}\widehat{v}(t_n,\xi_1)\widehat{v}(t_n,\xi_2)\widehat{v}(t_n,\xi_3)\,d(\xi_1)d(\xi_2)ds,
\end{align*}
which by further noting \eqref{formu-a4}:
\begin{align*}
\frac{1}{\xi_1}+\frac{1}{\xi_2}+\frac{1}{\xi_3}
&=\frac{1}{\xi_1}+\frac{1}{\xi_2}+\frac{1}{\xi_3}-\frac1\xi+\frac1\xi\\
&=\frac{\xi\xi_1\xi_2+\xi\xi_1\xi_3+\xi\xi_2\xi_3-\xi_1\xi_2\xi_3}{\xi\xi_1\xi_2\xi_3} +\frac1\xi\\
&=\frac13\frac{\alpha_4}{\xi\xi_1\xi_2\xi_3} +\frac1\xi,
\end{align*}
we find that for $\xi\neq0$,
\begin{subequations}\label{2.15}
\begin{align}
\mathcal F\eqref{10.40-3-2}(\xi)
=&\frac i{27}\int_0^\tau\!\!\int_{\xi=\xi_1+\xi_2+\xi_3\neq0} \frac{\alpha_4}{\xi\xi_1\xi_2\xi_3} \fe^{-i(t_n+s)\alpha_4}\widehat{v}(t_n,\xi_1)
\widehat{v}(t_n,\xi_2)\widehat{v}(t_n,\xi_3)\,d(\xi_1)d(\xi_2)ds\label{11.08-1}\\
 &+\frac{i}{9\xi}\int_0^\tau\!\!\int_{\xi=\xi_1+\xi_2+\xi_3\neq0} \fe^{-i(t_n+s)\alpha_4}\widehat{v}(t_n,\xi_1)
 \widehat{v}(t_n,\xi_2)\widehat{v}(t_n,\xi_3)\,d(\xi_1)d(\xi_2)ds.\label{11.08-2}
\end{align}
\end{subequations}
For \eqref{11.08-1}, it is direct to get
\begin{align*}
\eqref{11.08-1}=
&-\frac 1{27}\!\!\int_{\xi=\xi_1+\xi_2+\xi_3\neq0} \frac1{\xi\xi_1\xi_2\xi_3}\Big( \fe^{-it_{n+1}\alpha_4}-\fe^{-it_{n}\alpha_4}\Big)\widehat{v}(t_n,\xi_1)
\widehat{v}(t_n,\xi_2)\widehat{v}(t_n,\xi_3)\,d(\xi_1)d(\xi_2),
\end{align*}
which clearly can be transformed back explicitly to the physical space.
To further make (\ref{11.08-2}) explicitly integrable in the physical space, we just freeze the phase function in the integrand: $\fe^{-i(t_n+s)\alpha_4}\approx\fe^{-it_n\alpha_4}$, and simply obtain
\begin{align*}
\eqref{11.08-2}\approx
&\frac {i\tau}{9\xi}\int_{\xi=\xi_1+\xi_2+\xi_3\neq0} \fe^{-it_{n}\alpha_4}\widehat{v}(t_n,\xi_1)\widehat{v}(t_n,\xi_2)\widehat{v}(t_n,\xi_3)\,d(\xi_1)d(\xi_2).
\end{align*}
This leads to a total $O(\tau^2)$ truncation, but remarkably, thanks to (\ref{formu-a4}) and the factor $\frac{1}{\xi}$ in front, only one spatial derivative is lost in this approximation.
Combining the above two findings for (\ref{2.15}), we obtain that for $\xi\ne 0$,
\begin{align*}
\mathcal F\eqref{10.40-3-2}(\xi)
\approx&-\frac 1{27}\!\!\int_{\xi=\xi_1+\xi_2+\xi_3\neq0} \frac1{\xi\xi_1\xi_2\xi_3}\Big( \fe^{-it_{n+1}\alpha_4}-\fe^{-it_{n}\alpha_4}\Big)
\widehat{v}(t_n,\xi_1)\widehat{v}(t_n,\xi_2)\widehat{v}(t_n,\xi_3)\,d(\xi_1)d(\xi_2)\\
 &
+\frac {i\tau}{9\xi}\int_{\xi=\xi_1+\xi_2+\xi_3\neq0} \fe^{-it_{n}\alpha_4}\widehat{v}(t_n,\xi_1)\widehat{v}(t_n,\xi_2)\widehat{v}(t_n,\xi_3)\,d(\xi_1)d(\xi_2).
\end{align*}
In total, the inverse Fourier transform of the above together with \eqref{zero-mod} leads to
\begin{align*}
\eqref{10.40-3-2}
\approx&-\mathcal F\eqref{est:10.40-1}(0)+\frac1{27}\fe^{t_{n}\partial_x^3}\partial_x^{-1}
\Big(\fe^{-t_{n}\partial_x^3}\partial_x^{-1}v(t_n)\Big)^3\nonumber\\
&
-\frac1{27}\fe^{t_{n+1}\partial_x^3}\partial_x^{-1}
\Big(\fe^{-t_{n+1}\partial_x^3}\partial_x^{-1}v(t_n)\Big)^3-\frac{\tau}{9}\partial_x^{-1} \fe^{t_n\partial_x^3} \big(\fe^{-t_n\partial_x^3}v(t_n)\big)^3,
\end{align*}
and this completes our approximation of (\ref{10.40-3}) defined in the physical space.

Finally, combining the above with \eqref{est:10.40-1} and \eqref{est:10.40-2} and noting our projection operator $
\P $ defined in (\ref{P def}), we derive the following approximation from (\ref{10.40}) as
\begin{align*}
I_2(t_n)
\approx&\frac1{9}\P\Big[\fe^{t_{n+1}\partial_x^3}\Big(\fe^{-t_{n+1}\partial_x^3}
\partial_x^{-1}v(t_n)\cdot\partial_x^{-1}\big(\fe^{-t_{n+1}\partial_x^3}
\partial_x^{-1}v(t_n)\big)^2\Big)\Big]\\
&-\frac1{9}\P\Big[\fe^{t_{n+1}\partial_x^3}\Big(\fe^{-t_{n+1}\partial_x^3}
\partial_x^{-1}v(t_n)
\cdot\fe^{-\tau\partial_x^3}\partial_x^{-1}\big(\fe^{-t_{n}\partial_x^3}
\partial_x^{-1}v(t_n)\big)^2\Big)\Big]\\
 &+\frac1{27}\fe^{t_{n}\partial_x^3}\partial_x^{-1}
 \Big(\fe^{-t_{n}\partial_x^3}\partial_x^{-1}v(t_n)\Big)^3
-\frac1{27}\fe^{t_{n+1}\partial_x^3}\partial_x^{-1}
\Big(\fe^{-t_{n+1}\partial_x^3}\partial_x^{-1}v(t_n)\Big)^3\\
&-\frac\tau{9}\partial_x^{-1} \fe^{t_n\partial_x^3} \big(\fe^{-t_n\partial_x^3}v(t_n)\big)^3
+\frac\tau{6\pi} \partial_x^{-1}v(t_n)\int_\T \left(v(t_n)\right)^2\,dx,\quad n\geq0.
\end{align*}
This together with the equality for $I_1(t_n)$ in (\ref{est:15.15-1}) ends up as a first order numerical scheme for the mild solution $v^n=v^n(x)\approx v(t,x)$ of  (\ref{vt eq}): $n\geq0$,
\begin{align}
v^{n+1}=&v^n-\frac16\fe^{t_n\partial_x^3}\left(\fe^{-t_n\partial_x^3}
\partial_x^{-1}v^n\right)^2
+\frac16\fe^{t_{n+1}\partial_x^3}\left(\fe^{-t_{n+1}\partial_x^3}
\partial_x^{-1}v^n\right)^2\nonumber\\
&+\frac1{18}\P\left[\fe^{t_{n+1}\partial_x^3}\left(\fe^{-t_{n+1}\partial_x^3}
\partial_x^{-1}v^n\cdot \partial_x^{-1}\left(\fe^{-t_{n+1}\partial_x^3}\partial_x^{-1}v^n\right)^2\right)
\right]\nonumber\\
&-\frac1{18}\P\left[
\fe^{t_{n+1}\partial_x^3}\left(\fe^{-t_{n+1}\partial_x^3}\partial_x^{-1}v^n\cdot \fe^{-\tau\partial_x^3}\partial_x^{-1}\left(\fe^{-t_{n}\partial_x^3}
\partial_x^{-1}v^n\right)^2\right)\right]\nonumber\\
&+\frac1{54}\fe^{t_{n}\partial_x^3}\partial_x^{-1}\left(\fe^{-t_{n}\partial_x^3}
\partial_x^{-1}v^n\right)^3
-\frac1{54}\fe^{t_{n+1}\partial_x^3}\partial_x^{-1}\left(\fe^{-t_{n+1}\partial_x^3}
\partial_x^{-1}v^n\right)^3\nonumber\\
&+\frac\tau{12\pi} \partial_x^{-1}v^n\int_\T \left(v^n\right)^2\,dx
-\frac\tau {18}\partial_x^{-1} \fe^{t_n\partial_x^3} \left(\fe^{-t_n\partial_x^3}v^n\right)^3.\label{lem:LRI-1or}
\end{align}

By reversing the change of unknown (\ref{twist}) in (\ref{lem:LRI-1or}), we write down the scheme of the \emph{first order embedded low-regularity integrator (ELRI1)} for solving the KdV equation (\ref{model}): denoting
$u^n=u^n(x)\approx u(t_n,x)$ as the numerical solution,  for $n=0,1,\ldots$,
\begin{align}\label{def:LRI-1ord scheme}
u^{n+1}=&\fe^{-\tau\partial_x^3}u^n-\frac16\fe^{-\tau\partial_x^3}
\left(\partial_x^{-1}u^n\right)^2
+\frac16\left(\fe^{-\tau\partial_x^3}\partial_x^{-1}u^n\right)^2
+\frac1{18}\P\left[\fe^{-\tau\partial_x^3}\partial_x^{-1}u^n\cdot \partial_x^{-1}\left(\fe^{-\tau\partial_x^3}\partial_x^{-1}u^n\right)^2
\right]\nonumber\\
&-\frac1{18}\P\left[
\fe^{-\tau\partial_x^3}\partial_x^{-1}{u}^n\cdot \fe^{-\tau\partial_x^3}\partial_x^{-1}\left(\partial_x^{-1}u^n\right)^2\right]
+\frac1{54}\fe^{-\tau\partial_x^3}\partial_x^{-1}\left(\partial_x^{-1}u^n\right)^3
-\frac1{54}\partial_x^{-1}\left(\fe^{-\tau\partial_x^3}\partial_x^{-1}u^n\right)^3
\notag\\
&+\frac\tau{12\pi} \fe^{-\tau\partial_x^3} \partial_x^{-1}u^n\int_\T \left(u^n\right)^2\,dx
-\frac\tau {18}\partial_x^{-1} \fe^{-\tau\partial_x^3} \left(u^n\right)^3\nonumber\\
=:&ELRI1(u^n),
\end{align}
with $
\P $  in (\ref{P def}).

The above ELRI1 scheme (\ref{def:LRI-1ord scheme}) is clearly fully explicit and  efficient, since it is defined point-wisely in the physical space. All the involved pseudo-differentiation operators and the projection $\P$ in the scheme can be implemented by the Fourier pseudo-spectral method \cite{Shen}. Therefore, thanks to the fast Fourier transform (FFT), the computational cost per time step is $O(N\log N)$ with $N>0$ the number of spatial grids in practice. Also, it is clear that the ELRI1 (\ref{def:LRI-1ord scheme}) preserves the mean-value of solution of the KdV equation (\ref{model}) at the discrete level, i.e.
$$\int_\bT u^n(x)dx\equiv\int_\bT u_0(x)dx,\quad n=0,1,\ldots. $$

\begin{remark}
  The existing first order low-regularity integrator (LRI1) proposed in \cite{kdv-kath} reads
  $$u^{n+1}=\fe^{-\tau\partial_x^3}u^n-\frac16\fe^{-\tau\partial_x^3}
\left(\partial_x^{-1}u^n\right)^2
+\frac16\left(\fe^{-\tau\partial_x^3}\partial_x^{-1}u^n\right)^2,\quad n=0,1,\ldots.$$
Comparing to (\ref{def:LRI-1ord scheme}), our ELRI1 scheme  can be viewed as adding some correction terms to the above LRI1 in order to save the regularity in rough solution situation. Note that all the added correction terms can be implemented efficiently under FFT. Although the scheme of ELRI1 involves more terms than LRI1, with the gain in accuracy, ELRI1 is in fact much more efficient than  LRI1. This will be illustrated in Section \ref{sec:numerical}.
\end{remark}

\subsection{Second order scheme}
To further get the second-order accuracy, we allow ourself to ask for more regularity of the solution than the first order scheme. This is practically reasonable, because the accuracy order of the scheme  in time and in space need to be rather equal.

We follow the same derivation as for the first order scheme all the way till (\ref{2.15}). Note that the dropped term  in (\ref{v app1}) is a $O(\tau^3)$ truncation if two spatial derivatives are paid. Now in \eqref{11.08-2}:
$$\frac{i}{9\xi}\int_0^\tau\!\!\int_{\xi=\xi_1+\xi_2+\xi_3\neq0} \fe^{-i(t_n+s)\alpha_4}\widehat{v}(t_n,\xi_1)
 \widehat{v}(t_n,\xi_2)\widehat{v}(t_n,\xi_3)\,d(\xi_1)d(\xi_2)ds,$$
  we need to make a more accurate approximation of the phase function $\fe^{-i(t_n+s)\alpha_4}$.
Note that a direct higher order Taylor's expansion e.g.
$\fe^{-is\alpha_4}= 1-is\alpha_4+O(s^2\alpha_4^2)$ may not work well, since the explicit spatial derivative will come out which causes stability issue in the final scheme.
 To avoid the stability concern, we introduce a filtered mid-point rule to the linear part of the above Taylor's expansion, i.e. $$is\alpha_4\approx \frac{i\tau}{2}
 \alpha_4\fe^{-is\alpha_4},\quad 0\leq s\leq \tau,$$
  and we obtain the following approximation:
\begin{align*}
 &\eqref{11.08-2}\approx\frac {i}{9\xi}\int_0^\tau
\int_{\xi=\xi_1+\xi_2+\xi_3\neq0}\fe^{-it_n\alpha_4}\Big(1-\frac12 i\tau\alpha_4\>\fe^{-is\alpha_4}\Big) \widehat{v}(t_n,\xi_1)\widehat{v}(t_n,\xi_2)
\widehat{v}(t_n,\xi_3)\,d(\xi_1)d(\xi_2)\,ds.
\end{align*}
Then the right-hand side term is a $O(\tau^3)$ approximation with three additional derivatives required (see in detailed proof later), and it can be expressed explicitly in the physical space:
$$
-\frac\tau {9}\partial_x^{-1} \fe^{t_n\partial_x^3} \big(\fe^{-t_n\partial_x^3}{ v}(t_n)\big)^3
+\frac{\tau}{18}\fe^{t_n\partial_x^3}\partial_x^{-1}\Big(\fe^{-t_n\partial_x^3}{ v}(t_n)\Big)^3
  -\frac{\tau}{18}\fe^{t_{n+1}\partial_x^3}\partial_x^{-1}\Big(\fe^{-t_{n+1}\partial_x^3}{ v}(t_n)\Big)^3.
$$
 Then we can end up with a second order scheme for the mild solution $v^n=v^n(x)\approx v(t,x)$ in (\ref{vt eq}): $n\geq0$,
\begin{align}
v^{n+1}=&v^n-\frac16\fe^{t_n\partial_x^3}\left(\fe^{-t_n\partial_x^3}
\partial_x^{-1}{v}^n\right)^2
+\frac16\fe^{t_{n+1}\partial_x^3}\left(\fe^{-t_{n+1}\partial_x^3}
\partial_x^{-1}{v}^n\right)^2\nonumber\\
&+\frac1{18}\P\left[\fe^{t_{n+1}\partial_x^3}\left(\fe^{-t_{n+1}\partial_x^3}
\partial_x^{-1}{v}^n\cdot \partial_x^{-1}\left(\fe^{-t_{n+1}\partial_x^3}\partial_x^{-1}{v}^n\right)^2
\right)\right]\nonumber\\
&-\frac1{18}\P\left[
\fe^{t_{n+1}\partial_x^3}\left(\fe^{-t_{n+1}\partial_x^3}\partial_x^{-1}{v}^n\cdot \fe^{-\tau\partial_x^3}\partial_x^{-1}\left(\fe^{-t_{n}\partial_x^3}
\partial_x^{-1}{v}^n\right)^2\right)\right]\nonumber\\
&+\frac1{54}\fe^{t_{n}\partial_x^3}\partial_x^{-1}
\left(\fe^{-t_{n}\partial_x^3}\partial_x^{-1}{v}^n\right)^3
-\frac1{54}\fe^{t_{n+1}\partial_x^3}\partial_x^{-1}
\left(\fe^{-t_{n+1}\partial_x^3}\partial_x^{-1}{v}^n\right)^3
\nonumber\\
&+\frac\tau{12\pi} \partial_x^{-1}{ v}^n\int_\T \left(v^n\right)^2\,dx
-\frac\tau {36}\partial_x^{-1} \fe^{t_n\partial_x^3}\left(\fe^{-t_n\partial_x^3}{ v}^n\right)^3
-\frac{\tau}{36}\fe^{t_{n+1}\partial_x^3}\partial_x^{-1}
  \left(\fe^{-t_{n+1}\partial_x^3}{ v}^n\right)^3.
  \label{lem:LRI-2or}
\end{align}
In fact, compared with (\ref{lem:LRI-1or}), the second order scheme above can be viewed as simply adding two (or one, since the last term in (\ref{lem:LRI-1or}) is emerged with the first correction) more correction term(s) to the first order scheme (\ref{lem:LRI-1or}).

Similarly as before by reversing the change of unknown, the \emph{second order embedded low-regularity integrator (ELRI2)} for solving the KdV equation (\ref{model}) reads: denoting
$u^n=u^n(x)\approx u(t_n,x)$ as the numerical solution,  for $n=0,1,\ldots$,
\begin{align}\label{def:LRI-2ord scheme}
{u}^{n+1}=
ELRI1(u^n)+\frac{\tau}{36}\fe^{-\tau\partial_x^3}\partial_x^{-1}\Big({ u}^n\Big)^3
  -\frac{\tau}{36}\partial_x^{-1}\Big(\fe^{-\tau\partial_x^3}
  {u}^n\Big)^3,
\end{align}
with $ELRI1(\cdot)$ defined in (\ref{def:LRI-1ord scheme}).

The above ELRI2 scheme is also explicit and efficient (almost the same) as ELRI1,  and it  preserves the mean-value of the solution of the KdV equation (\ref{model}) at the discrete level as well.

\subsection{Main convergence results}
Now, we state the convergence theorems of the proposed ELRI1 method (\ref{def:LRI-1ord scheme}) and ELRI2 method (\ref{def:LRI-2ord scheme}) as the \textbf{main results} of the paper. The proofs will be given in the next two sections.

\begin{theorem}\label{thm:convergence}
Let $u^n$ be the numerical solution from the ELRI1 scheme \eqref{def:LRI-1ord scheme} for the KdV equation (\ref{model}) with $\widehat{u_0}(0)=0$ up to some fixed time $T>0$. Under assumption that $u_0\in H^{\gamma+1}(\bT)$ for some $\gamma>\frac12$,  there exist constants $\tau_0,C>0$ such that for  any $0<\tau\leq\tau_0$ we have
\begin{equation}\label{main result}
  \|u(t_n,\cdot)-u^n\|_{H^\gamma}\leq C\tau,\quad n=0,1\ldots,\frac{T}{\tau},
\end{equation}
 where the constants $\tau_0$ and $C$ depend only on $T$ and $\|u\|_{L^\infty((0,T);H^{\gamma+1})}$.
\end{theorem}

\begin{theorem}\label{thm:convergence-2ord}
Let $u^n$ be the numerical solution  from the ELRI2 scheme  \eqref{def:LRI-2ord scheme} for the KdV equation (\ref{model}) with $\widehat{u_0}(0)=0$ up to some fixed time $T>0$. Under assumption that $u_0\in H^{\gamma+3}(\bT)$ for some $\gamma\ge 0$,  there exist constants $\tau_0,C>0$ such that for  any $0<\tau\leq\tau_0$ we have
\begin{equation}\label{main result-2}
  \|u(t_n,\cdot)-u^n\|_{H^\gamma}\leq C\tau^2,\quad n=0,1\ldots,\frac{T}{\tau},
\end{equation}
 where the constants $\tau_0$ and $C$ depend only on $T$ and $\|u\|_{L^\infty((0,T);H^{\gamma+3})}$.
\end{theorem}

A direct result of the above convergence theorems with our shifting technique (\ref{shift}) gives the convergence of ELRIs for the case of $u_0$ with general average value in (\ref{model}), i.e. $\widehat{u_0}(0)\in\bR$.

\begin{corollary}\label{cor:ELRI}
  Let $u^n$ be defined in ELRI1 \eqref{def:LRI-1ord scheme} up to some fixed time $T>0$. Assuming in (\ref{model}) that $u_0\in H^{\gamma+1}(\bT)$ for some $\gamma>\frac12$, there exist constants $\tau_0,C>0$ such that for any $0<\tau\leq\tau_0$,
$$
  \left\|u(t_n,\cdot)-u^n\left(\cdot+t_n\widehat{u_0}(0)\right)
  +\widehat{u_0}(0)\right\|_{H^\gamma}\leq C\tau,\quad n=0,1\ldots,\frac{T}{\tau},
$$
 where $\tau_0$ and $C$ depend only on $T$ and $\|u\|_{L^\infty((0,T);H^{\gamma+1})}$. Moreover if $u_0\in H^{\gamma+3}(\bT)$ for some $\gamma\ge 0$, then with $u^n$ from ELRI2  \eqref{def:LRI-2ord scheme}, we further have
 $$
  \left\|u(t_n,\cdot)-u^n\left(\cdot+t_n\widehat{u_0}(0)\right)
  +\widehat{u_0}(0)\right\|_{H^\gamma}\leq C\tau^2,\quad n=0,1\ldots,\frac{T}{\tau},
$$
 where $\tau_0$ and $C$ depend only on $T$ and $\|u\|_{L^\infty((0,T);H^{\gamma+3})}$.
\end{corollary}

From the above convergence theorems, we see that the ELIR1 can get the first order optimal convergence rate with only one additional spatial derivatives, and the ELIR2 reaches the second order convergence rate with three additional   spatial derivatives, which remarkably release the regularity requirement of the existing schemes in the literature including the direct low-regularity integrators \cite{Schratznew,kdv-kath,WuZhao-1} and the classical methods.

\begin{remark}
In order to focus on the error of the numerical integrators, the above two main convergence theorems are for the semi-discretized (in time) schemes.  The error estimates of the fully discretized schemes with different kinds of spatial discretizations will be the target of our future studies.
\end{remark}

\begin{remark}
We shall show later by numerical results that the condition $\gamma>\frac12$ in Theorem  \ref{thm:convergence} or Corollary \ref{cor:ELRI} is not essential. It is just a technique condition for the rigorous proof given in the two coming sections. We claim here that $\gamma\ge 0$ in Theorem  \ref{thm:convergence} is also sufficient based on numerical experiments. The proof of this however, would rather require much more technical analysis. To avoid blurring the theme of the paper, we shall address this in a separated future work. 

\end{remark}
%

\begin{remark}
Although we restrict to consider the KdV equation (\ref{model}) on the torus $\bT$ for simplicity in this paper, our methods and convergence results also hold in the whole space case, i.e. $x\in\bR$, with little modifications.
\end{remark}

\section{The first-order convergence analysis}\label{sec:1ord-proof}
In this section, we will establish the rigorous proof of the convergence result: Theorem \ref{thm:convergence} for the ELRI1 scheme (\ref{def:LRI-1ord scheme}).

We shall still assume that the average of the initial value in (\ref{model}) is zero.
Since the change of unknown (\ref{twist}) is isometric in the Sobolev space $H^\gamma$,  we will prove the convergence result (\ref{main result}) for the mild solution $v$ (\ref{v eq}) and $v^n$ (\ref{lem:LRI-1or}). To do so, we follow the classical framework of convergence proof \cite{Lubich,lownls} by investigating the local error and stability. For short, we shall denote the functional space
$$L_t^\infty H^{s_0}_x=L^\infty((0,T);H^{s_0}(\bT))$$
for any $s_0\geq0$. Here $T>0$ is the final time of the numerical solution in Theorem \ref{thm:convergence}.

First of all, based on the derivation of the ELRI1 scheme in Section \ref{sec:scheme}, we rewrite the numerical mild solution $v^{n+1}$ from (\ref{lem:LRI-1or}) for $n=0,\ldots,\frac{T}{\tau}-1$ as
\begin{align}
v^{n+1}
 =&v^n+
 \frac12\int_0^\tau\fe^{(t_n+s)\partial_x^3}
 \partial_x\left(\fe^{-(t_n+s)\partial_x^3}v^n\right)^2\,ds\notag\\
 &+ \frac12\int_0^\tau\fe^{(t_n+s)\partial_x^3}
 \partial_x\left(\fe^{-(t_n+s)\partial_x^3}v^n\cdot \fe^{-(t_n+s)\partial_x^3}F_n(v^n,s)\right)\,ds
 +\frac1{18}\int_0^\tau \mathcal{A}_n(v^n)(s) \,ds.\label{LRI-discrete-revised}
\end{align}
Here $F_n $ is defined in (\ref{def-Fn}), and the operator $\mathcal{A}_n$ is defined in Fourier space as
\begin{equation}\label{A def}
 \mathcal F\big(\mathcal{A}_n(f_1,f_2,f_3)\big)(t,\xi)\\
  := \left\{ \aligned
    &0,\qquad \mbox{if}\quad \xi=0,\\
    &(i\xi)^{-1}\int_{\xi=\xi_1+\xi_2+\xi_3} \Big(\fe^{-i(t_{n}+t)\alpha_4}-\fe^{-it_{n}\alpha_4}\Big)\\
    &\qquad\qquad\cdot \widehat f_1(\xi_1)\widehat f_2(\xi_2)\widehat f_3(\xi_3)\,(d\xi_1)(d\xi_2), \qquad \mbox{if}\quad  \xi\ne 0,
   \endaligned
  \right.
 \end{equation}
where for simplicity, we denote $\mathcal{A}_n(f)=\mathcal{A}_n(f,f,f)$.

Taking the difference between \eqref{LRI-discrete-revised} and (\ref{nest v}), we find
\begin{align}\label{v-L-Phi}
v^{n+1}-v(t_{n+1})=\mathcal{L}^n+\Phi^n\big(v^n\big)-\Phi^n\big(v(t_n)\big),\quad n=0,\ldots,\frac{T}{\tau}-1,
\end{align}
where we denote the local truncation error as
\begin{align}
\mathcal{L}^n:=
&\frac12 \int_0^\tau\fe^{(t_n+s)\partial_x^3}
 \partial_x\left(\fe^{-(t_n+s)\partial_x^3}\Big(v(t_n)+\frac12F_n\big(v(t_n),s\big)\Big)\right)^2\,ds\nonumber\\
 &-\frac12\int_0^\tau\fe^{(t_n+s)\partial_x^3}
 \partial_x\left(\fe^{-(t_n+s)\partial_x^3}v(t_n+s)\right)^2\,ds\nonumber\\
&-\frac18\int_0^\tau\fe^{(t_n+s)\partial_x^3}
 \partial_x\left(\fe^{-(t_n+s)\partial_x^3}F_n\big(v(t_n),s\big)\right)^2\,ds
 +\frac1{18}\int_0^\tau \mathcal{A}_n\big(v(t_n)\big)(s) \,ds,\label{local error}
\end{align}
$F_n$ is defined in \eqref{def-Fn}, and the numerical propagator as
\begin{align}
\Phi^n(v):=
&v+\frac12 \int_0^\tau\fe^{(t_n+s)\partial_x^3}
 \partial_x\left(\fe^{-(t_n+s)\partial_x^3}v\right)^2\,ds\nonumber\\
&+\frac12\int_0^\tau\fe^{(t_n+s)\partial_x^3}
 \partial_x\left(\fe^{-(t_n+s)\partial_x^3}v\cdot \fe^{-(t_n+s)\partial_x^3}F_n(v,s)\right)\,ds+\frac1{18}\int_0^\tau \mathcal{A}_n(v)(s) \,ds.
 \label{propagator}
\end{align}

Then we shall analyze the local error and the stability of the numerical propagator in a sequel.  The whole proof goes in an induction or Bootstrap manner for the boundedness of the numerical solution.

\subsection{Some tool lemmas}\label{subsec2}For the convenience of analysis, we first introduce some notations and tool lemmas, some of which are employed from our previous work \cite{WuZhao-1}. We use $A\lesssim B$ or $B\gtrsim A$ to denote the statement that $A\leq CB$ for some absolute constant $C>0$ which may
vary from line to line but independent of $\tau$ or $n$.

As a basic tool to deal with the fractional derivatives, we will frequently call the following Kato-Ponce inequality \cite{BoLi-KatoPonce,Kato-Ponce, Li-KatoPonce,WuZhao-1}.

\begin{lemma}\label{lem:kato-Ponce}(Kato-Ponce inequality) The following inequalities hold:
\begin{itemize}
  \item[(i)]
  For any $\gamma\ge 0, \gamma_1>\frac12$, $f,g\in H^{\gamma}\cap  H^{\gamma_1}$, then
\begin{align*}
\|J^\gamma (fg)\|_{L^2}\lesssim \|f\|_{H^\gamma}\|g\|_{H^{\gamma_1}}+ \|f\|_{H^{\gamma_1}}\|g\|_{H^{\gamma}}.
\end{align*}
In particular, if $\gamma>\frac12$, then
\begin{align*}
\|J^\gamma (fg)\|_{L^2}\lesssim \|f\|_{H^\gamma}\|g\|_{H^{\gamma}}.
\end{align*}
  \item[(ii)]
For any  $\gamma\ge 0, \gamma_1>\frac12$, $f\in H^{\gamma+\gamma_1} ,g\in H^{\gamma}$, then
\begin{align*}
\|J^\gamma (fg)\|_{L^2}\lesssim \|f\|_{H^{\gamma+\gamma_1}}\|g\|_{H^{\gamma}}.
\end{align*}
\end{itemize}
\end{lemma}

Moreover for the stability analysis, we will need the following estimate of the commutator $[A,B]:=AB-BA$.
\begin{lemma}\label{lem:kato-Ponce-2}
Let $\gamma\ge 0, \gamma_1>\frac12$, then the following inequalities hold:
\begin{itemize}
  \item[(i)]
For any  $f\in H^{\gamma}\cap H^{\gamma_1}$ and $g\in H^{\gamma+1}\cap H^{\gamma_1+1}$:
\begin{align*}
\big\|[J^\gamma\partial_x, g] f\big\|_{L^2}\lesssim \|f\|_{H^\gamma}\|g\|_{H^{\gamma_1+1}}+\|f\|_{H^{\gamma_1}}\|g\|_{H^{\gamma+1}}.
\end{align*}
  \item[(ii)]
For any  $f\in H^{\gamma}$ and $g\in H^{\gamma+\gamma_1+1}$:
\begin{align*}
\big\|[J^\gamma, g]\partial_x f\big\|_{L^2}\lesssim  \|f\|_{H^\gamma}\|g\|_{H^{\gamma_1+1}}+\|f\|_{L^2}\|g\|_{H^{\gamma+\gamma_1+1}}.
\end{align*}
\end{itemize}
\end{lemma}
\begin{proof}
(i) Taking the Fourier transform on $[J^\gamma\partial_x, g] f$, we get
\begin{align*}
\mathcal F\Big([J^\gamma\partial_x , g]f\Big)(\xi)=i\int_{\xi=\xi_1+\xi_2}\big(\langle \xi\rangle^\gamma\xi-\langle \xi_1\rangle^\gamma\xi_1\big) \widehat f(\xi_1)\widehat g(\xi_2)\,(d\xi_1).
\end{align*}
We assume that $\widehat f$ and $\widehat g$ are positive, otherwise one may replace them by $|\widehat f|$ and $|\widehat g|$. Noting that
$$
\big|\langle \xi\rangle^\gamma\xi-\langle \xi_1\rangle^\gamma\xi_1\big|
\lesssim |\xi_2|\big(\langle \xi_1\rangle^\gamma+\langle \xi_2\rangle^\gamma\big),
$$
 we have
\begin{align*}
\quad\Big|\mathcal F\Big([J^\gamma\partial_x, g] f\Big)(\xi)\Big|
&\lesssim \int_{\xi=\xi_1+\xi_2}|\xi_2|\big(\langle \xi_1\rangle^\gamma+\langle \xi_2\rangle^\gamma\big) \widehat f(\xi_1)\widehat g(\xi_2)\,(d\xi_1)\\
&=\mathcal F\Big(\langle \nabla\rangle^\gamma f\cdot |\nabla| g\Big)(\xi)+\mathcal F\Big(f\cdot\langle \nabla^\gamma\rangle|\nabla| g\Big)(\xi).
\end{align*}
Hence, by Plancherel's identity and Sobolev's inequality,
\begin{align}
\big\|[J^\gamma\partial_x, g] f\big\|_{L^2}&
=\sqrt{2\pi}\Big\|\mathcal F\Big([J^\gamma\partial_x, g] f\Big)(\xi)\Big\|_{L^2}\notag\\
&\lesssim \big\|\langle \nabla\rangle^\gamma f\cdot |\nabla| g\big\|_{L^2}+\big\|f\cdot\langle \nabla\rangle^\gamma|\nabla| g\big\|_{L^2}\notag\\
&\lesssim \big\|\langle \nabla\rangle^\gamma f\big\|_{L^2}\big\||\nabla|g\big\|_{L^\infty}+\big\| f\big\|_{L^\infty}\big\|\langle \nabla\rangle^\gamma|\nabla|g\big\|_{L^2}\label{1154}\\
&\lesssim \|f\|_{H^\gamma}\|g\|_{H^{\gamma_1+1}}+\|f\|_{H^{\gamma_1}}\|g\|_{H^{\gamma+1}}.\notag
\end{align}
This gives the desired result (i).

(ii) Replacing the estimate \eqref{1154} by the following:
\begin{align*}
\big\|[J^\gamma\partial_x, g] f\big\|_{L^2}
&\lesssim \big\|\langle \nabla\rangle^\gamma f\big\|_{L^2}\big\||\nabla|g\big\|_{L^\infty}+\big\| f\big\|_{L^2}\big\|\langle \nabla\rangle^\gamma|\nabla|g\big\|_{L^\infty}\\
&\lesssim \|f\|_{H^\gamma}\|g\|_{H^{\gamma_1+1}}+\|f\|_{L^2}\|g\|_{H^{\gamma+\gamma_1+1}}.
\end{align*}
Then we have the desired result (ii), and the lemma is proved.
\end{proof}

Based on the above inequalities,
we can deduce some estimates as follows.
\begin{lemma}\label{lm3.5}
The following inequalities hold:
\begin{itemize}
  \item[(i)]
  For any $\gamma\ge 0, \gamma_1>\frac12$, $f\in H^{\gamma}\cap H^{\gamma_1},g\in H^{\gamma+1}\cap H^{\gamma_1+1}$, then
\begin{align}
\big\langle J^\gamma\partial_x(fg),J^\gamma f\big\rangle\lesssim \|f\|_{H^\gamma}^2\|g\|_{H^{\gamma_1+1}}+\|f\|_{H^\gamma}\|f\|_{H^{\gamma_1}}\|g\|_{H^{\gamma+1}}.\label{Comm-1}
\end{align}
In particular, if $\gamma>\frac12$, then
\begin{align}
\big\langle J^\gamma\partial_x(fg),J^\gamma f\big\rangle\lesssim \|f\|_{H^\gamma}^2\|g\|_{H^{\gamma+1}}.\label{Comm-2}
\end{align}
  \item[(ii)]
  For any $\gamma\ge 0, \gamma_1>\frac12$, $f\in H^{\gamma}\cap H^{\gamma_1+1}, g\in H^{\gamma+\gamma_1+1}$, then
\begin{align}
\big\langle J^\gamma\partial_x(fg),J^\gamma f\big\rangle\lesssim \|f\|_{H^\gamma}^2\|g\|_{H^{\gamma+\gamma_1+1}}.\label{Comm-3}
\end{align}
\end{itemize}
\end{lemma}
\begin{proof}
Directly, we have
\begin{align*}
\big\langle J^\gamma\partial_x(fg),J^\gamma f\big\rangle=\big\langle J^\gamma \partial_x f\cdot g,J^\gamma f\big\rangle
+\big\langle\big[J^\gamma\partial_x,g\big]\> f,J^\gamma f\big\rangle.
\end{align*}
For the first term on the right-hand side, by using integration-by-parts, it is equal to
\begin{align*}
-\frac12\int \partial_xg\big|J^\gamma f\big|^2\,dx.
\end{align*}
Therefore, we have the estimate
\begin{align*}
\big|\big\langle J^\gamma \partial_x f\cdot g,J^\gamma f\big\rangle \big|
\lesssim  \big\|g\big\|_{H^{1+\gamma_1}}\big\|f\big\|_{H^\gamma}^2,
\end{align*}
for any $\gamma_1>\frac12$. For the second term, by Lemma \ref{lem:kato-Ponce-2} (i), we have
\begin{align*}
\big|\big\langle\big[J^\gamma\partial_x ,g\big]\>f,J^\gamma f\big\rangle\big|
\lesssim &  \big\|\big[J^\gamma\partial_x  ,g\big]\>f\big\|_{L^2} \big\| f\big\|_{H^\gamma}\\
\lesssim& \|f\|_{H^\gamma}^2\|g\|_{H^{\gamma_1+1}}+\|f\|_{H^\gamma}\|f\|_{H^{\gamma_1}}\|g\|_{H^{\gamma+1}}.
\end{align*}
Combining the estimates above, we get the estimate \eqref{Comm-1}. If $\gamma>\frac12$, then choosing $\gamma_1=\gamma$, we get \eqref{Comm-2}.
Furthermore, the estimate \eqref{Comm-3} follows by using Lemma \ref{lem:kato-Ponce-2} (ii) instead.
\end{proof}

As a consequence of (\ref{lem:1-form}) and the above Kato-Ponce inequality, we have the following lemma.
\begin{lemma}\label{lem:bi-est}
The following inequalities hold:
\begin{itemize}
\item[(i)] Let the space-time functions $f(t,x),g(t,x)\in L^\infty_tH^{\gamma}_x$ and $\partial_tf(t,x),\partial_tg(t,x)\in L^\infty_tH^{\gamma-1}_x$ with $\widehat{f}(t,0)=\widehat{g}(t,0)=0$ for $\gamma>\frac12$ and $t\in [0,\tau]$, then the following inequality holds for any $t_n\geq0$:
\begin{align}
&\Big\|\int_0^\tau \fe^{(t_n+s)\partial_x^3}\partial_x\big(\fe^{-(t_n+s)\partial_x^3}f(s)\cdot \fe^{-(t_n+s)\partial_x^3}g(s)\big)\,ds\Big\|_{H^\gamma}\notag\\
\lesssim & \sqrt\tau\|f\|_{L^\infty_tH^\gamma_x}\|g\|_{L^\infty_tH^\gamma_x}+\tau \Big(\big\|\partial_tf\big\|_{L^\infty_tH^{\gamma-1}_x}\big\|g\big\|_{L^\infty_tH^{\gamma-1}_x}
+\big\|f\big\|_{L^\infty_tH^{\gamma-1}_x}\big\|\partial_tg\big\|_{L^\infty_tH^{\gamma-1}_x}\Big).\label{bi-est-2}
\end{align}
\item[(ii)] Let the space-time functions $f(t,x),g(t,x)\in L^\infty_tH^{\gamma-1}_x$ and $\partial_tf(t,x),\partial_tg(t,x)\in L^\infty_tH^{\gamma-1}_x$ with $\widehat{f}(t,0)=\widehat{g}(t,0)=0$ for $\gamma>\frac12$ and $t\in [0,\tau]$, then the following inequality holds for any $t_n\geq0$:
\begin{align}
&\Big\|\int_0^\tau \fe^{(t_n+s)\partial_x^3}\partial_x\big(\fe^{-(t_n+s)\partial_x^3}f(s)\cdot \fe^{-(t_n+s)\partial_x^3}g(s)\big)\,ds\Big\|_{H^\gamma}\notag\\
\lesssim & \|f\|_{L^\infty_tH^{\gamma-1}_x}\|g\|_{L^\infty_tH^{\gamma-1}_x}+\tau \Big(\big\|\partial_tf\big\|_{L^\infty_tH^{\gamma-1}_x}\big\|g\big\|_{L^\infty_tH^{\gamma-1}_x}
+\big\|f\big\|_{L^\infty_tH^{\gamma-1}_x}\big\|\partial_tg\big\|_{L^\infty_tH^{\gamma-1}_x}\Big).\label{bi-est-2}
\end{align}
\end{itemize}
\end{lemma}
\begin{proof}
(i) We use the formula in Lemma \ref{lem:1-form} (i) and denote
\begin{align*}
I=&\int_0^\tau \fe^{(t_n+s)\partial_x^3}\partial_x\left(\fe^{-(t_n+s)\partial_x^3}f(s)\cdot \fe^{-(t_n+s)\partial_x^3}g(s)\right)\,ds,\\
I_1=&\frac13 \fe^{t_{n+1}\partial_x^3}\left(\fe^{-t_{n+1}\partial_x^3}\partial_x^{-1}f(\tau)\cdot \fe^{-t_{n+1}\partial_x^3}\partial_x^{-1}g(\tau)\right),\\
I_2=&  -\frac13\fe^{t_{n}\partial_x^3}\left(\fe^{-t_{n}\partial_x^3}\partial_x^{-1}f(0)\cdot \fe^{-t_{n}\partial_x^3}\partial_x^{-1}g(0)\right),\\
I_3=  &-\frac13 \int_0^\tau \fe^{(t_{n}+s)\partial_x^3}\Big(\fe^{-(t_{n}+s)\partial_x^3}\partial_x^{-1}\partial_tf(s)\cdot \fe^{-(t_{n}+s)\partial_x^3}\partial_x^{-1}g(s)\\
  &\qquad \qquad\qquad\qquad+\fe^{-(t_{n}+s)\partial_x^3}\partial_x^{-1}f(s)\cdot \fe^{-(t_{n}+s)\partial_x^3}\partial_x^{-1}\partial_tg(s)\Big)\,ds,
\end{align*}
then
$
I=I_1+I_2+I_3.
$
Thus, we have
\begin{align*}
\|I\|_{H^\gamma}^2=&\langle J^\gamma I, J^\gamma I_1\rangle+\langle J^\gamma I, J^\gamma I_2\rangle+\langle J^\gamma I, J^\gamma I_3\rangle\\
\le &\big|\langle J^\gamma I, J^\gamma I_1\rangle\big|+\big|\langle J^\gamma I, J^\gamma I_2\rangle\big|+\frac12 \|I\|_{H^\gamma}^2+\frac12\|I_3\|_{H^\gamma}^2.
\end{align*}
Hence, we get
\begin{align*}
\|I\|_{H^\gamma}^2
\le &2\big|\langle J^\gamma I, J^\gamma I_1\rangle\big|+2\big|\langle J^\gamma I, J^\gamma I_2\rangle\big|+\|I_3\|_{H^\gamma}^2.
\end{align*}
For the term $\langle J^\gamma I, J^\gamma I_1\rangle$, by using the integration-by-parts, we get
\begin{align*}
\langle J^\gamma I, J^\gamma I_1\rangle
= &
-\frac13\int_0^\tau \Big\langle\fe^{(t_n+s)\partial_x^3}J^\gamma\left(\fe^{-(t_n+s)\partial_x^3}f(s)\cdot \fe^{-(t_n+s)\partial_x^3}g(s)\right), \\
&\qquad \fe^{t_{n+1}\partial_x^3}J^\gamma\partial_x\left(\fe^{-t_{n+1}\partial_x^3}\partial_x^{-1}f(\tau)\cdot \fe^{-t_{n+1}\partial_x^3}\partial_x^{-1}g(\tau)\right)\Big\rangle\,ds.
\end{align*}
Then by H\"older's inequality and Lemma \ref{lem:kato-Ponce} (i),  we get that for any $\gamma>\frac12$,
\begin{align*}
\big|\langle J^\gamma I, J^\gamma I_1\rangle\big|
\lesssim &
\int_0^\tau \Big\|\fe^{(t_n+s)\partial_x^3}J^\gamma\left(\fe^{-(t_n+s)\partial_x^3}f(s)\cdot \fe^{-(t_n+s)\partial_x^3}g(s)\right)\Big\|_{L^2}\\
&\qquad \cdot\Big\|\fe^{t_{n+1}\partial_x^3}J^\gamma\partial_x
\left(\fe^{-t_{n+1}\partial_x^3}\partial_x^{-1}f(\tau)\cdot \fe^{-t_{n+1}\partial_x^3}\partial_x^{-1}g(\tau)\right)\Big\|_{L^2}\,ds\\
\lesssim &
\int_0^\tau \|f(s)\|_{H^\gamma}^2\|g(s)\|_{H^\gamma}^2\,ds.
\end{align*}
Hence, we obtain
$$
\big|\langle J^\gamma I, J^\gamma I_1\rangle\big|\lesssim \tau\|f\|_{L^\infty_tH^\gamma_x}^2\|g\|_{L^\infty_tH^\gamma_x}^2.
$$
Similarly, we get
$$
\big|\langle J^\gamma I, J^\gamma I_2\rangle\big|
\lesssim \tau\|f\|_{L^\infty_tH^\gamma_x}^2\|g\|_{L^\infty_tH^\gamma_x}^2.
$$
For the term $\|I_3\|_{H^\gamma}$, by Lemma \ref{lem:kato-Ponce} (i), we get
\begin{align}\label{est:I3}
\|I_3\|_{H^\gamma}\lesssim \tau \Big(\big\|\partial_tf\big\|_{L^\infty_tH^{\gamma-1}_x}\big\|g\big\|_{L^\infty_tH^{\gamma-1}_x}
+\big\|f\big\|_{L^\infty_tH^{\gamma-1}_x}\big\|\partial_tg\big\|_{L^\infty_tH^{\gamma-1}_x}\Big).
\end{align}
Combining with the two estimates above, we give the proof of the result (i).

(ii) By  the triangle inequality directly, we have
$$
\|I\|_{H^\gamma}\le  \|I_1\|_{H^\gamma}+\|I_2\|_{H^\gamma}+\|I_3\|_{H^\gamma}.
$$
From Lemma \ref{lem:kato-Ponce} (i), we get that
$$
\|I_1\|_{H^\gamma}, \|I_2\|_{H^\gamma}\lesssim  \|f\|_{L^\infty_tH^{\gamma-1}_x}\|g\|_{L^\infty_tH^{\gamma-1}_x}.
$$
These last two estimates above combining with \eqref{est:I3} give the result (ii).
\end{proof}

Also,  we have the following estimates for the introduced operator $\mathcal{A}_n$ in  (\ref{A def}).
\begin{lemma}\label{lem:An}
The following inequalities hold:
\begin{itemize}
\item[(i)]Let $\gamma_0>\frac12$, and $f_1,f_2,f_3\in H^{\gamma_0}$, then  for any  $t\geq0$,
$$
\big\|\mathcal{A}_n(f_1,f_2,f_3) (t)\big\|_{H^{\gamma_0}}
\lesssim \big\|f_1\big\|_{H^{\gamma_0}}\big\|f_2\big\|_{H^{\gamma_0}}
\big\|f_3\big\|_{H^{\gamma_0}}.
$$

\item[(ii)]
Let $\gamma_0>\frac12$, and $f_1,f_2,f_3\in H^{\gamma_0+1}$, then for any $0\le t\le \tau$,
$$
\big\|\mathcal{A}_n(f_1,f_2,f_3) (t)\big\|_{H^{\gamma_0}}
\lesssim \tau \big\|f_1\big\|_{H^{\gamma_0+1}}\big\|f_2\big\|_{H^{\gamma_0+1}}\big\|f_3\big\|_{H^{\gamma_0+1}}.
$$
\end{itemize}
\end{lemma}
\begin{proof}
We assume that $\widehat f_j$ for $j=1,2,3$ are positive, otherwise one may replace them by $|\widehat f_j|$.

(i) We only consider the non-trivial case $|\xi|\ge 1$.  Then we have
\begin{align*}
&|\xi|^{\gamma_0}\big|\mathcal F\big(\mathcal{A}_n(f_1,f_2,f_3)\big)(t,\xi)\big|\\
\lesssim &|\xi|^{\gamma_0}\int_{\xi=\xi_1+\xi_2+\xi_3}\widehat f_1(\xi_1)\widehat f_2(\xi_2)\widehat f_3(\xi_3)\,(d\xi_1)(d\xi_2)\\
=& \mathcal F\Big(|\nabla|^{\gamma_0}\big(f_1 f_2 f_3\big)\Big).
\end{align*}
Then by Plancherel's identity and Lemma \ref{lem:kato-Ponce} (i), we obtain that for $\gamma_0>\frac12$,
$$
\big\|\mathcal{A}_n(f_1,f_2,f_3)(t)\big\|_{H^{\gamma_0}}
\lesssim \big\|f_1\big\|_{H^{\gamma_0}}\big\|f_2\big\|_{H^{\gamma_0}}\big\|f_3\big\|_{H^{\gamma_0}}.
$$

(ii) Note that
$$
\Big|\fe^{-i(t_{n}+t)\alpha_4}-\fe^{-it_{n}\alpha_4}\Big|\le t|\alpha_4|,
$$
and so we have
$$
\big|\mathcal F\big(\mathcal{A}_n(f_1,f_2,f_3)\big)(t,\xi)\big|\le t|\xi|^{-1}\int_{\xi=\xi_1+\xi_2+\xi_3} |\alpha_4|\widehat f_1(\xi_1)\widehat f_2(\xi_2)\widehat f_3(\xi_3)\,(d\xi_1)(d\xi_2).
$$
Using \eqref{formu-a4}, we further get
\begin{align*}
&|\xi|^{\gamma_0}\big|\mathcal F\big(\mathcal{A}_n(f_1,f_2,f_3)\big)(t,\xi)\big|\\
\lesssim &t\int_{\xi=\xi_1+\xi_2+\xi_3} |\xi|^{\gamma_0}\big(|\xi_1||\xi_2|+|\xi_1||\xi_3|+|\xi_2||\xi_3|\big)\widehat f_1(\xi_1)\widehat f_2(\xi_2)\widehat f_3(\xi_3)\,(d\xi_1)(d\xi_2)\\
& + t\int_{\xi=\xi_1+\xi_2+\xi_3} |\xi|^{\gamma_0-1}|\xi_1||\xi_2||\xi_3|\widehat f_1(\xi_1)
\widehat f_2(\xi_2)\widehat f_3(\xi_3)\,(d\xi_1)(d\xi_2).
\end{align*}
By symmetry, we may assume that $|\xi_1|\ge |\xi_2|\ge |\xi_3|$, then
\begin{align*}
&|\xi|^{\gamma_0}\big|\mathcal F\big(\mathcal{A}_n(f_1,f_2,f_3)\big)(t,\xi)\big|\\
\lesssim &t\int_{\xi=\xi_1+\xi_2+\xi_3,|\xi_1|\ge |\xi_2|\ge |\xi_3|}|\xi_1|^{1+\gamma_0}|\xi_2|\widehat f_1(\xi_1)\widehat f_2(\xi_2)\widehat f_3(\xi_3)\,(d\xi_1)(d\xi_2)\\
=& t \mathcal F\Big(|\nabla|^{1+\gamma_0}f_1\cdot |\nabla|f_2\cdot f_3\Big).
\end{align*}
Therefore, by Plancherel's identity and Lemma \ref{lem:kato-Ponce} (ii), we obtain that for any $\gamma_1>\frac12$,
$$
\big\|\mathcal{A}_n(f_1,f_2,f_3)(t)\big\|_{H^{\gamma_0}}
\lesssim t\big\|f_1\big\|_{H^{\gamma_0+1}}\big\| f_2\big\|_{H^{\gamma_1+1}}\big\|f_3\big\|_{H^{\gamma_1}}.
$$
Since $\gamma_0>\frac12$,  by choosing $\gamma_1=\gamma_0$, we get the desired result.
\end{proof}

\subsection{Local error}\label{local 1st}
For the local truncation error $\mathcal{L}^n$ defined in (\ref{local error}), we have the following estimate.
\begin{lemma}\label{lem:local-error} Let $\gamma >\frac12$ and $0< \tau \lesssim 1$, then
$$
\big\|\mathcal{L}^n\big\|_{H^\gamma}\le C\tau^2,\quad n=0,\ldots,\frac{T}{\tau}-1,
$$
 where the constant  $C$ depends only on $\|u\|_{L^\infty_tH^{\gamma+1}_x}$.
\end{lemma}
\begin{proof}
We split $\mathcal{L}^n$ into the following three parts as
$$
\mathcal{L}^n=\mathcal{L}^n_1+\mathcal{L}^n_2+\mathcal{L}^n_3,
$$
where
\begin{subequations}
\begin{align}
\mathcal{L}^n_1=
&\frac12 \int_0^\tau\fe^{(t_n+s)\partial_x^3}
 \partial_x\left[\fe^{-(t_n+s)\partial_x^3}\left(v(t_n)
 +\frac12F_n\left(v(t_n),s\right)\right)\right]^2ds\notag\\
 &-\frac12\int_0^\tau\fe^{(t_n+s)\partial_x^3}
 \partial_x\left(\fe^{-(t_n+s)\partial_x^3}v(t_n+s)\right)^2ds,\label{L1}\\
\mathcal{L}^n_2=
&-\frac18\int_0^\tau\fe^{(t_n+s)\partial_x^3}
 \partial_x\left(\fe^{-(t_n+s)\partial_x^3}F_n\big(v(t_n),s\big)\right)^2ds,\label{L2}
\\
\mathcal{L}^n_3=
&\frac1{18}\int_0^\tau \mathcal{A}_n\big(v(t_n)\big)(s) \,ds.\notag
\end{align}
\end{subequations}

For $\mathcal{L}^n_1$, we have that
\begin{align*}
\mathcal{L}^n_1=
&-\frac12 \int_0^\tau\fe^{(t_n+s)\partial_x^3}
 \partial_x\Bigg[\fe^{-(t_n+s)\partial_x^3}\left(v(t_{n}+s)-v(t_n)-\frac12 F_n\big(v(t_n),s\big)\right)\\
 &\qquad\qquad\ \cdot \fe^{-(t_n+s)\partial_x^3}\left(v(t_{n}+s)+v(t_n)+\frac12 F_n\big(v(t_n),s\big)\right) \Bigg]ds.
\end{align*}
Then by Lemma \ref{lem:bi-est} (ii), we get
\begin{align*}
&\big\|\mathcal{L}^n_1\big\|_{H^{\gamma}}\\
\lesssim
&\big\|v(t_{n}+s)-v(t_n)-\frac12 F_n\big(v(t_n),s\big)\big\|_{L^\infty_tH^{\gamma-1}_x}\big\|v(t_{n}+s)+
v(t_n)+\frac12 F_n\big(v(t_n),s\big)\big\|_{L^\infty_tH^{\gamma-1}_x}\\
 &+\tau \left\|\partial_s\Big(v(t_{n}+s)-v(t_n)-\frac12 F_n\big(v(t_n),s\big)\Big)\right\|_{L^\infty_tH^{\gamma-1}_x}
 \big\|v(t_{n}+s)+v(t_n)+\frac12 F_n\big(v(t_n),s\big)\big\|_{L^\infty_tH^{\gamma-1}_x}\\
& +\tau\big\|v(t_{n}+s)-v(t_n)-\frac12 F_n\big(v(t_n),s\big)\big\|_{L^\infty_tH^{\gamma-1}_x}
\left\|\partial_s\Big(v(t_{n}+s)+v(t_n)+\frac12 F_n\big(v(t_n),s\big)\Big)\right\|_{L^\infty_tH^{\gamma-1}_x}.
\end{align*}
Note that from \eqref{vt eq},
\begin{align*}
&\quad v(t_{n}+s)-v(t_n)-\frac12 F_n\left(v(t_n),s\right)\notag\\
&=\frac12\int_0^s \fe^{(t_{n}+t)\partial_x^3}\partial_x\left[\fe^{-(t_{n}+t)\partial_x^3}
\left(v(t_n+t)-v(t_n)\right)
\cdot \fe^{-(t_{n}+t)\partial_x^3}\left(v(t_n+t)+v(t_n)\right)\right]dt \\
&= \frac14\int_0^s\!\! \int_0^{t} \fe^{(t_{n}+t)\partial_x^3}\partial_x\left[\fe^{-(t-\rho)\partial_x^3}
\partial_x\left(\fe^{-(t_{n}+\rho)\partial_x^3}v(t_n+\rho)\right)^2
\cdot \fe^{-(t_{n}+t)\partial_x^3}\left(v(t_n+t)+v(t_n)\right)\right]d\rho dt.
\end{align*}
Hence,   by Lemma \ref{lem:kato-Ponce} (i),  we get that for any $s\le \tau$,
\begin{align}
\left\|v(t_n+s)-v(t_n)-\frac12 F_n\left(v(t_n),s\right)\right\|_{H^{\gamma-1}}
\lesssim \tau^2 \left\|v\right\|_{L^\infty_t H^{\gamma+1}_x}^3.\label{est:v-v-F}
\end{align}
Similarly, we have
\begin{align*}
&\quad \partial_s\left(v(t_{n}+s)-v(t_n)-\frac12 F_n\big(v(t_n),s\big)\right)\\
&=\frac14\int_0^{s} \fe^{-(t_{n}+s)\partial_x^3}\partial_x\left[\fe^{-(s-\rho)\partial_x^3}
\partial_x\left(\fe^{-(t_{n}+\rho)\partial_x^3}\left(v(t_n+\rho)\right)^2
\cdot \fe^{-(t_{n}+s)\partial_x^3}\left(v(t_n+s)+v(t_n)\right)\right)\right]d\rho.
\end{align*}
This implies that  for any  $s\le \tau$,
$$
\left\|\partial_s\left(v(t_{n}+s)-v(t_n)-\frac12 F_n\big(v(t_n),s\big)\right)\right\|_{L^\infty_tH^{\gamma-1}_x}
\lesssim \tau \left\|v\right\|_{L^\infty_t H^{\gamma+1}_x}^3.
$$
Moreover, from the definition \eqref{def-Fn}, we have that for any $0\le s\le \tau$,
\begin{align}
\left\|\partial_s F_n\left(v(t_n),s\right)\right\|_{H^{\gamma-1}}
\lesssim \left\|v\right\|_{L^\infty_t H^{\gamma}_x}^2, 
\quad \left\|F_n\left(v(t_n),s\right)\right\|_{H^{\gamma-1}}
\lesssim \tau\left\|v\right\|_{L^\infty_t H^{\gamma}_x}^2.\label{est:Fn}
\end{align}
Similarly, by \eqref{vt eq}, we have  that for any $0\le s\le \tau$,
\begin{align}
\left\|\partial_s v(t_n+s)\right\|_{H^{\gamma-1}}
\lesssim \left\|v\right\|_{L^\infty_t H^{\gamma}_x}^2.\label{est:vn}
\end{align}
Hence, using these estimates, we obtain that for any $s\le \tau$,
\begin{align*}
\left\|v(t_{n}+s)+v(t_n)+\frac12 F_n\big(v(t_n),s\big)\right\|_{L^\infty_tH^{\gamma-1}_x}\lesssim
\left\|v\right\|_{L^\infty_t H^{\gamma}_x}+\left\|v\right\|_{L^\infty_t H^{\gamma}_x}^2,
\end{align*}
and
\begin{align*}
\left\|\partial_s\left(v(t_{n}+s)+v(t_n)+\frac12 F_n\big(v(t_n),s\big)\right)\right\|_{L^\infty_tH^{\gamma-1}_x}
\lesssim
\left\|v\right\|_{L^\infty_t H^{\gamma}_x}^2.
\end{align*}
Therefore, combining with the estimates above, we have
\begin{align}
\left\|\mathcal{L}^n_1\right\|_{H^{\gamma}}
\lesssim \tau^2\left( \left\|v\right\|_{L^\infty_tH^{\gamma+1}_x}^4+\left\|v\right\|_{L^\infty_t H^{\gamma+1}_x}^5\right).\label{est:Ln-1}
\end{align}

For $\mathcal{L}^n_2$, by Lemma \ref{lem:bi-est} (ii) again, we get
\begin{align}
\left\|\mathcal{L}^n_2\right\|_{H^{\gamma}}
\lesssim \left\|F_n\left(v(t_n),s\right)\right\|_{H^{\gamma-1}}^2
+ \tau\left\|F_n\left(v(t_n),s\right)\right\|_{H^{\gamma-1}}\left\|\partial_s F_n\left(v(t_n),s\right)\right\|_{H^{\gamma-1}}.
\end{align}
Hence, by  \eqref{est:Fn}, we obtain that
\begin{align}
\left\|\mathcal{L}^n_2\right\|_{H^{\gamma}}
\lesssim \tau^2 \left\|v\right\|_{L^\infty_tH^{\gamma}_x}^4.\label{est:Ln-2}
\end{align}

For $\mathcal{L}_3^n$, from Lemma \ref{lem:An} (ii), we get
\begin{align}
\left\|\mathcal{L}^n_3\right\|_{H^{\gamma}}
\lesssim \tau^2 \left\|v\right\|_{L^\infty_t H^{\gamma+1}_x}^3.\label{est:Ln-3}
\end{align}

Combing \eqref{est:Ln-1},  \eqref{est:Ln-2}  and \eqref{est:Ln-3}, the lemma is proved.
\end{proof}

\subsection{Stability}\label{stable 1st}
For the numerical propagator $\Phi^n$ defined in (\ref{propagator}), we have the following stability result.
\begin{lemma}\label{lem:stability} Let $\gamma >\frac12$, then for $n=0,\ldots,\frac{T}{\tau}-1$,
$$
\left\|\Phi^n(v^n)-\Phi^n(v(t_n))\right\|_{H^\gamma}\le \left(1+ C\tau+C\sqrt\tau\left\|v^n-v(t_n)\right\|_{H^\gamma}+C\tau\left\|v^n-v(t_n)\right\|_{H^\gamma}^4\right)\left\|v^n-v(t_n)\right\|_{H^\gamma},
$$
 where the constant  $C$ depends only on $\|u\|_{L^\infty_tH^{\gamma+1}_x}$.
\end{lemma}
\begin{proof} Note that for $n=0,\ldots,\frac{T}{\tau}-1$,
\begin{align*}
&\Phi^n(v^n)-\Phi^n(v(t_n))
 =v^n-v(t_n)+\Phi^n_1+\Phi^n_2+\Phi^n_3,
\end{align*}
where we denote
\begin{align*}
\Phi^n_1=&\frac12 \int_0^\tau
 \fe^{(t_n+s)\partial_x^3}\partial_x\left[\left(\fe^{-(t_n+s)\partial_x^3}
 v^n\right)^2-\left(\fe^{-(t_n+s)\partial_x^3}
 v(t_n)\right)^2\right]ds,\\
\Phi^n_2=&\frac12\int_0^\tau\fe^{(t_n+s)\partial_x^3}
 \partial_x\Big[\fe^{-(t_n+s)\partial_x^3}v^n\cdot \fe^{-(t_n+s)\partial_x^3}F_n(v^n,s)\\
 &\qquad\qquad\qquad\qquad-\fe^{-(t_n+s)\partial_x^3}v(t_n)\cdot \fe^{-(t_n+s)\partial_x^3}F_n(v(t_n),s)\Big]ds,\\
 \Phi^n_3=&\frac1{18}\int_0^\tau \left[\mathcal{A}_n\left(v^n\right)(s)
 -\mathcal{A}_n\left(v(t_n)\right)(s)\right]ds.
\end{align*}
For short, we denote the error $e_n=v^n-v(t_n)$ within this lemma, then
\begin{align*}
\left\|\Phi\left(v(t_n)\right)-\Phi\left(v^n\right)\right\|_{H^\gamma}^2
\le& \|e_n\|_{H^\gamma}^2+2\left\langle J^\gamma\Phi^n_1,J^\gamma e_n\right\rangle
+2\|e_n\|_{H^\gamma}\left\|\Phi^n_2\right\|_{H^\gamma}+2\|e_n\|_{H^\gamma}\left\|\Phi^n_3\right\|_{H^\gamma}
\\
&+3\left\|\Phi^n_1\right\|_{H^\gamma}^2+3\left\|\Phi^n_2\right\|_{H^\gamma}^2
+3\left\|\Phi^n_3\right\|_{H^\gamma}^2.
\end{align*}

Firstly, we rewrite $\Phi^n_1$ as
\begin{align*}
\Phi^n_1=&\frac12 \int_0^\tau
 \fe^{(t_n+s)\partial_x^3}\partial_x\Big(\fe^{-(t_n+s)\partial_x^3}
 e_n\Big)^2ds+ \int_0^\tau
 \fe^{(t_n+s)\partial_x^3}\partial_x\Big(\fe^{-(t_n+s)\partial_x^3}
 e_n\cdot\fe^{-(t_n+s)\partial_x^3}v(t_n)\Big)ds,
 \end{align*}
 and thus
\begin{subequations}\label{Phi-1}
\begin{align}
\left\langle J^\gamma\Phi^n_1,J^\gamma e_n\right\rangle
=&\frac12
 \left\langle  \int_0^\tau J^\gamma\fe^{-(t_n+s)\partial_x^3}\partial_x\left(\fe^{-(t_n+s)\partial_x^3}
 e_n\right)^2 ds,J^\gamma e_n\right\rangle \label{Phi-1-1}\\
 &+ \int_0^\tau
 \left\langle J^\gamma \partial_x\left(\fe^{-(t_n+s)\partial_x^3}
 e_n\cdot\fe^{-(t_n+s)\partial_x^3}v(t_n)\right),J^\gamma \fe^{-(t_n+s)\partial_x^3}e_n\right\rangle ds\label{Phi-1-2}
\end{align}
\end{subequations}
By using Lemma \ref{lem:bi-est} (i), we get
\begin{align}
|\eqref{Phi-1-1}|
\lesssim
&
  \left\|\int_0^\tau \fe^{-(t_n+s)\partial_x^3}\partial_x\left(\fe^{-(t_n+s)\partial_x^3}
 e_n\right)^2 ds\right\|_{H^\gamma} \|e_n\|_{H^\gamma}
\lesssim
\sqrt\tau \|e_n\|_{H^\gamma}^3.
\end{align}
Furthermore, using Lemma \ref{lm3.5} (i), we get
\begin{align*}
|\eqref{Phi-1-2}|
\lesssim
&\tau \left\|v(t_n)\right\|_{H^{\gamma+1}}\left\|e_n\right\|_{H^\gamma}^2.
\end{align*}
Combining with the two finding on \eqref{Phi-1} above, we obtain that
\begin{align}
\left|\left\langle J^\gamma\Phi^n_1,J^\gamma e_n\right\rangle\right|
\lesssim
&\tau\left\|e_n\right\|_{H^\gamma}^2+\sqrt\tau \|e_n\|_{H^\gamma}^3.\label{est:Phi-1-f}
\end{align}

For $\left\|\Phi^n_1\right\|_{H^\gamma}$, from Lemma \ref{lem:bi-est} (i), we have
\begin{align}
\left\|\Phi_1^n\right\|_{H^\gamma}\lesssim
\sqrt\tau \Big(\left\|v\right\|_{L^\infty_tH^{\gamma}_x}\left\|e_n\right\|_{H^{\gamma}}
+\left\|e_n\right\|_{H^{\gamma}}^2\Big).\label{est:Phi-1}
\end{align}

Secondly,  we rewrite $\Phi^n_2$ as
\begin{subequations}
\begin{align}
\Phi^n_2=&\frac12\int_0^\tau\fe^{(t_n+s)\partial_x^3}
 \partial_x\left[\fe^{-(t_n+s)\partial_x^3}e_n\cdot \fe^{-(t_n+s)\partial_x^3}F_n(v(t_n),s)\right]ds\label{13.51-1}\\
 &+\frac12\int_0^\tau\fe^{(t_n+s)\partial_x^3}
 \partial_x\left[\fe^{-(t_n+s)\partial_x^3}e_n\cdot \fe^{-(t_n+s)\partial_x^3}\left(F_n(v^n,s)-F_n(v(t_n),s)\right)\right]ds\label{13.51-2}\\
 &+\frac12\int_0^\tau\fe^{(t_n+s)\partial_x^3}
 \partial_x\left[\fe^{-(t_n+s)\partial_x^3}v(t_n)\cdot \fe^{-(t_n+s)\partial_x^3}\left(F_n(v^n,s)-F_n(v(t_n),s)\right)\right]ds.\label{13.51-3}
\end{align}
\end{subequations}
For \eqref{13.51-1}, using Lemma \ref{lem:bi-est} (ii), we have
\begin{align*}
\left\|\eqref{13.51-1}\right\|_{H^\gamma}
\lesssim &\left\|e_n\right\|_{H^{\gamma-1}}\left\|F_n(v(t_n),\cdot)\right\|_{L^\infty_tH^{\gamma-1}_x}
+\tau \left\|e_n\right\|_{H^{\gamma-1}}\left\|\partial_tF_n(v(t_n),\cdot)\right\|_{L^\infty_tH^{\gamma-1}_x}.
\end{align*}
Hence, using  \eqref{est:Fn}, we get
\begin{align}\label{est:13.51-1}
\big\|\eqref{13.51-1}\big\|_{H^\gamma}
\lesssim &\tau\|v\|_{L^\infty_tH^{\gamma+1}_x}^2\left\|e_n\right\|_{H^{\gamma}}.
\end{align}
For \eqref{13.51-2}, using Lemma \ref{lem:bi-est} (ii) again,  we have
\begin{align*}
\big\|\eqref{13.51-2}\big\|_{H^\gamma}
\lesssim &\left\|e_n\right\|_{H^{\gamma-1}}\left\|F_n(v^n,\cdot)-F_n(v(t_n),\cdot)\right\|_{L^\infty_tH^{\gamma-1}_x}\\
& +\tau \left\|e_n\right\|_{H^{\gamma-1}}
\left\|\partial_tF_n(v^n,\cdot)-\partial_tF_n(v(t_n),\cdot)\right\|_{L^\infty_tH^{\gamma-1}_x}.
\end{align*}
From \eqref{def-Fn} and Lemma \ref{lem:kato-Ponce} (i),
\begin{align*}
\left\|F_n(v^n,\cdot)-F_n(v(t_n),\cdot)\right\|_{L^\infty_tH^{\gamma-1}_x}
\lesssim & \int_0^\tau\left\|\left(\fe^{-(t_n+t)\partial_x^3}v^n\right)^2
-\left(\fe^{-(t_n+t)\partial_x^3}v(t_n)\right)^2\right\|_{H^{\gamma}}\,dt\\
\lesssim &\tau\left\|e_n\right\|_{H^{\gamma}}^2+\tau\|v\|_{L^\infty_tH^{\gamma}_x}\left\|e_n\right\|_{H^{\gamma}};\\
\left\|\partial_tF_n(v^n,\cdot)-\partial_tF_n(v(t_n),\cdot)\right\|_{L^\infty_tH^{\gamma-1}_x}
\lesssim & \left\|\left(\fe^{-(t_n+t)\partial_x^3}v^n\right)^2
-\left(\fe^{-(t_n+t)\partial_x^3}v(t_n)\right)^2\right\|_{L^\infty_tH^{\gamma}_x}\\
\lesssim &\left\|e_n\right\|_{H^{\gamma}}^2+\|v\|_{L^\infty_tH^{\gamma}_x}\|e_n\|_{H^{\gamma}}.
\end{align*}
Using these two estimates, we get
\begin{align}\label{est:13.51-2}
\big\|\eqref{13.51-2}\big\|_{H^\gamma}
\lesssim &\tau\left\|e_n\right\|_{H^{\gamma}}^3+\tau\|v\|_{L^\infty_t H^{\gamma}_x}\left\|e_n\right\|_{H^{\gamma}}^2.
\end{align}
Similarly as for \eqref{13.51-2}, we have
\begin{align}\label{est:13.51-3}
\big\|\eqref{13.51-3}\big\|_{H^\gamma}
\lesssim &\tau\|v\|_{L^\infty_tH^{\gamma}}\left\|e_n\right\|_{H^{\gamma}}^2
+\tau\|v\|_{L^\infty_tH^{\gamma}_x
}^2\left\|e_n\right\|_{H^{\gamma}}.
\end{align}
Together with the estimates in \eqref{est:13.51-1}, \eqref{est:13.51-2} and \eqref{est:13.51-3}, we get
\begin{align}
\left\|\Phi_2^n\right\|_{H^\gamma}\le
C\tau \left(\left\|e_n\right\|_{H^{\gamma}}+\left\|e_n\right\|_{H^{\gamma}}^3\right),\label{est:Phi-2}
\end{align}
 where the constant  $C$ depends only on $\|u\|_{L^\infty_tH^{\gamma+1}_x}$.

For $\big\|\Phi_3^n\big\|_{H^\gamma}$, we use Lemma \ref{lem:An} (i) to yield
\begin{align}
\left\|\Phi_3^n\right\|_{H^\gamma}\le
C\tau \left(\left\|e_n\right\|_{H^{\gamma}}+\left\|e_n\right\|_{H^{\gamma}}^3\right),\label{est:Phi-3}
\end{align}
 where the constant  $C$ depends only on $\|u\|_{L^\infty_tH^{\gamma}_x}$.

Combining with \eqref{est:Phi-1-f}, \eqref{est:Phi-1}, \eqref{est:Phi-2} and \eqref{est:Phi-3}, we obtain
\begin{align*}
\left\|\Phi\left(v(t_n)\right)-\Phi\left(v^n\right)\right\|_{H^\gamma}^2
\le& \|e_n\|_{H^\gamma}^2+C\tau \left(\left\|e_n\right\|_{H^{\gamma}}^2
+\left\|e_n\right\|_{H^{\gamma}}^6\right)+\sqrt\tau \|e_n\|_{H^\gamma}^3.
\end{align*}
Applying the inequality $\sqrt{1+a}\le 1+a$ for any $a>0$, we obtain the desired result.
\end{proof}

\subsection{Proof of Theorem \ref{thm:convergence}}
Now, combining the local error estimate and the stability result, we give the proof of Theorem \ref{thm:convergence}.
From Lemma \ref{lem:local-error} and Lemma \ref{lem:stability}, there exits a constant $C>0$ such that for $0<\tau\leq1$, we have
\begin{align}
\left\|v(t_{n+1})-v^{n+1}\right\|_{H^\gamma}
\le & C_1\tau^2+\left(1+ C_2\tau+C_3\sqrt\tau\left\|v^n-v(t_n)\right\|_{H^\gamma}+C_4\tau\left\|v^n-v(t_n)\right\|_{H^\gamma}^4\right)\notag\\
&\quad \cdot\left\|v^n-v(t_n)\right\|_{H^\gamma},\ n=0,1,\ldots,\frac{T}{\tau}-1,\label{iteration}
\end{align}
where $C_j$ for $j=1,\cdots,4 $ depend on $\|u\|_{L^\infty_tH^{\gamma+1}_x}$.
 Now we claim that there exists some $\tau_0>0$ (to be determined) such that for any $\tau\in (0,\tau_0]$ and  any $ n=0,1,\ldots,\frac{T}{\tau}$,
\begin{align}\label{Hy}
\left\|v(t_{n})-v^{n}\right\|_{H^\gamma}
\le C_1 \tau^2 \sum\limits_{j=0}^n(1+2C_2\tau)^j.
\end{align}
We prove it by induction. Note that \eqref{Hy} trivially  holds for $n=0$.  Now we assume that it holds 
till some $0\le n_0\le \frac{T}{\tau}-1$, i.e.
\begin{align}\label{Hy-1}
\left\|v(t_{n})-v^{n}\right\|_{H^\gamma}
\le C_1 \tau^2 \sum\limits_{j=0}^n(1+2C_2\tau)^j, \mbox{ for any } 0\le n\le n_0.
\end{align}
From \eqref{Hy-1}, we have that for any $ 0\le n\le n_0$,
\begin{align}\label{bound-v-vn}
\left\|v(t_{n})-v^{n}\right\|_{H^\gamma}
\le C_5\tau,
\end{align}
where $C_5=2C_1C_2^{-1}\fe^{2C_2T}$. Then by \eqref{iteration}, we find
\begin{align*}
\left\|v(t_{n_0+1})-v^{n_0+1}\right\|_{H^\gamma}
\le & C_1\tau^2+\left(1+ C_2\tau+C_3C_5\tau^\frac32+C_4C_5^4\tau^5\right)
\cdot C_1 \tau^2 \sum\limits_{j=0}^{n_0}(1+2C_2\tau)^j.
\end{align*}
Choose $\tau_0>0$ such that
$$
C_3C_5\tau_0^\frac12+C_4C_5^4\tau_0^4\le C_2,
$$
then for any $\tau\in (0,\tau_0]$, we obtain that
\begin{align*}
\left\|v(t_{n_0+1})-v^{n_0+1}\right\|_{H^\gamma}
\le & C_1\tau^2+\left(1+ 2C_2\tau\right)
\cdot C_1 \tau^2 \sum\limits_{j=0}^{n_0}(1+2C_2\tau)^j\\
= & C_1\tau^2+
 C_1 \tau^2 \sum\limits_{j=0}^{n_0}(1+2C_2\tau)^{j+1}
<C_1 \tau^2\sum\limits_{j=0}^{n_0+1}(1+2C_2\tau)^{j}.
\end{align*}
This finishes the induction and proves \eqref{Hy}. 
Hence, we get \eqref{bound-v-vn} for any $n=0,1,\ldots,\frac{T}{\tau}$,  and  Theorem \ref{thm:convergence} is proved.\qed



\section{The second-order convergence analysis}\label{sec:2ord-proof}
In this section, we give the rigorous proof of the second convergence result: Theorem \ref{thm:convergence-2ord} for the proposed ELRI2 scheme (\ref{def:LRI-2ord scheme}). Similarly as before,  we still have the zero-average of the initial value in (\ref{model}) and prove the convergence result (\ref{main result-2}) for the mild solution $v$ (\ref{v eq}) and $v^n$ (\ref{lem:LRI-2or}). We shall adopt some of the notations from the previous section.

Based on the derivation of the ELRI2 scheme in Section \ref{sec:scheme}, we can rewrite $v^{n+1}$ from (\ref{lem:LRI-2or}) as
\begin{align}
v^{n+1}
 =&{v}^n+
 \frac12\int_0^\tau\fe^{(t_n+s)\partial_x^3}
 \partial_x\left(\fe^{-(t_n+s)\partial_x^3}{ v}^n\right)^2\,ds\notag\\
 &+ \frac12\int_0^\tau\fe^{(t_n+s)\partial_x^3}
 \partial_x\left(\fe^{-(t_n+s)\partial_x^3}{v}^n\cdot \fe^{-(t_n+s)\partial_x^3}F_n({ v}^n,s)\right)\,ds
 +\frac1{18}\int_0^\tau \widetilde{\mathcal{A}}_n({ v}^n)(s) \,ds,\label{LRI-discrete-revised-2}
\end{align}
where
we define the operator $\widetilde{\mathcal{A}}_n$ in Fourier space as
\begin{equation}\label{def:tilde-An}
 \mathcal F\big(\widetilde{\mathcal{A}}_n(f_1,f_2,f_3)\big)(t,\xi)\\
  := \left\{ \aligned
    &0,\qquad \mbox{if}\quad \xi=0,\\
    &(i\xi)^{-1}\int_{\xi=\xi_1+\xi_2+\xi_3} \Big(\fe^{-it\alpha_4}-1+\frac12i\tau\alpha_4\fe^{-it\alpha_4}\Big)\fe^{-it_{n}\alpha_4}\\
    &\qquad\qquad\cdot \widehat f_1(\xi_1)\widehat f_2(\xi_2)\widehat f_3(\xi_3)\,(d\xi_1)(d\xi_2), \qquad \mbox{if}\quad  \xi\ne 0,
   \endaligned
  \right.
 \end{equation}
and similarly as before $\widetilde{\mathcal{A}}_n(f)=\widetilde{\mathcal{A}}_n(f,f,f)$ for short.
Define the local error
\begin{align}
\mathcal{\widetilde{L}}^n:=
&\frac12 \int_0^\tau\fe^{(t_n+s)\partial_x^3}
 \partial_x\left(\fe^{-(t_n+s)\partial_x^3}\left(v(t_n)
 +\frac12F_n\left(v(t_n),s\right)\right)\right)^2\,ds\label{local 2nd}\\
 &-\frac12\int_0^\tau\fe^{(t_n+s)\partial_x^3}
 \partial_x\left(\fe^{-(t_n+s)\partial_x^3}v(t_n+s)\right)^2ds\nonumber\\
&-\frac18\int_0^\tau\fe^{(t_n+s)\partial_x^3}
 \partial_x\left(\fe^{-(t_n+s)\partial_x^3}F_n\left(v(t_n),s\right)\right)^2ds
 +\frac1{18}\int_0^\tau \widetilde{\mathcal{A}}_n\left(v(t_n)\right)(s) ds,\nonumber
\end{align}
and the propagator
\begin{align}
\widetilde{\Phi}^n(v):=
&v+\frac12 \int_0^\tau\fe^{(t_n+s)\partial_x^3}
 \partial_x\left(\fe^{-(t_n+s)\partial_x^3}v\right)^2ds\label{stable 2nd}\\
&+\frac12\int_0^\tau\fe^{(t_n+s)\partial_x^3}
 \partial_x\left(\fe^{-(t_n+s)\partial_x^3}v\cdot \fe^{-(t_n+s)\partial_x^3}F_n(v,s)\right)ds+\frac1{18}\int_0^\tau \widetilde{\mathcal{A}}_n(v)(s)\,ds, \nonumber
\end{align}
then by taking the difference between (\ref{nest v}) and (\ref{lem:LRI-2or}), we find
\begin{align*}
v^{n+1}-v(t_{n+1})=\mathcal{\widetilde{L}}^n+
\widetilde{\Phi}^n\left(v^n\right)-\widetilde{\Phi}^n\left(v(t_n)\right).
\end{align*}

In the next two subsections, we shall consider the estimates on $\mathcal{\widetilde{L}}^n$ and $\widetilde{\Phi}^n\big(v^n\big)-\widetilde{\Phi}^n\big(v(t_n)\big)$, respectively. Since they are quite similar as given in the Sections \ref{local 1st}\&\ref{stable 1st}, we will focus on the differences.



First of all, we update the estimates of the defined new operator $\widetilde{\mathcal{A}}_n$ in (\ref{def:tilde-An}) as follows.
\begin{lemma}\label{lem:An-2}
The following inequalities hold:
\begin{itemize}
\item[(i)]
Let $\gamma_0>\frac12$, and $f_1,f_2,f_3\in H^{\gamma_0}$, then
$$
\left\|\int_0^\tau\widetilde{\mathcal{A}}_n(f_1,f_2,f_3)(t)\,dt\right\|_{H^{\gamma_0}}
\lesssim \tau\left\|f_1\right\|_{H^{\gamma_0}}\left\|f_2\right\|_{H^{\gamma_0}}\left\|f_3\right\|_{H^{\gamma_0}}.
$$

\item[(ii)]
Let $\gamma_0\ge 0$, and $f_1,f_2,f_3\in H^{\gamma_0+3}$, then
$$
\left\|\int_0^\tau\widetilde{\mathcal{A}}_n(f_1,f_2,f_3)(t)\,dt\right\|_{H^{\gamma_0}}
\lesssim \tau^3 \left\|f_1\right\|_{H^{\gamma_0+3}}\left\|f_2\right\|_{H^{\gamma_0+3}}\left\|f_3\right\|_{H^{\gamma_0+3}}.
$$
\end{itemize}
\end{lemma}
\begin{proof}
(i) Note that when $\xi\ne 0$, by integrating in time we get
\begin{align*}
 &\int_0^\tau \mathcal F\big(\widetilde{\mathcal{A}}_n(f_1,f_2,f_3)\big)(t,\xi)\,dt
  -\int_0^\tau \mathcal F\big(\mathcal{A}_n(f_1,f_2,f_3)\big)(t,\xi)\,dt\\
  =&-\frac\tau2(i\xi)^{-1}\int_{\xi=\xi_1+\xi_2+\xi_3} \left(\fe^{-it_{n+1}\alpha_4}-\fe^{-it_{n}\alpha_4}\right) \widehat{f_1}(\xi_1)\widehat{f_2}(\xi_2)\widehat{f_3}(\xi_3)\,(d\xi_1)(d\xi_2).
 \end{align*}
Hence, we have
\begin{align*}
 \int_0^\tau\widetilde{\mathcal{A}}_n(f_1,f_2,f_3)(s)\,ds
  =&\int_0^\tau\mathcal{A}_n(f_1,f_2,f_3)(s)\,ds
  +\frac\tau2\fe^{t_n\partial_x^3}\partial_x^{-1}
  \left(\fe^{-t_n\partial_x^3}f_1\cdots\fe^{-t_n\partial_x^3}f_3\right)\\
  &-\frac{\tau}{2}\fe^{t_{n+1}\partial_x^3}\partial_x^{-1}
  \left(\fe^{-t_{n+1}\partial_x^3}f_1\cdots\fe^{-t_{n+1}\partial_x^3}f_3\right).
 \end{align*}
 For the second and the third terms, by Lemma \ref{lem:kato-Ponce} (i), we have that for any $\gamma_0>\frac12$,
 $$
 \left\|\fe^{t_n\partial_x^3}\partial_x^{-1}\left(\fe^{-t_n\partial_x^3}f_1
 \cdots\fe^{-t_n\partial_x^3}f_3\right)\right\|_{H^{\gamma_0}}
 \lesssim  \left\|f_1\right\|_{H^{\gamma_0}}\left\|f_2\right\|_{H^{\gamma_0}}\left\|f_3\right\|_{H^{\gamma_0}},
 $$
and
$$
 \left\|\fe^{t_{n+1}\partial_x^3}\partial_x^{-1}\left(\fe^{-t_{n+1}\partial_x^3}
 f_1\cdots\fe^{-t_{n+1}\partial_x^3}f_3\right)\right\|_{H^{\gamma_0}}
 \lesssim  \left\|f_1\right\|_{H^{\gamma_0}}\left\|f_2\right\|_{H^{\gamma_0}}\left\|f_3\right\|_{H^{\gamma_0}}.
 $$
Hence, these two estimates combining with Lemma \ref{lem:An} (i) give the desired result.

(ii)
We only consider the case when $|\xi|\ge 1$. Moreover, we may assume that $\widehat f_j$ for $j=1,2,3$ are positive, otherwise one may replace them by $|\widehat f_j|$.
Noting that
\begin{align*}
\int_0^\tau  \left(\fe^{-it\alpha_4}-1+\frac12i\tau\alpha_4\fe^{-it\alpha_4}\right)dt
=\int_0^\tau  \left(\fe^{-it\alpha_4}-1+it\alpha_4+\frac12i\tau\alpha_4
\left(\fe^{-it\alpha_4}-1\right)\right)dt,
\end{align*}
 we have
$$
\left|\int_0^\tau  \left(\fe^{-it\alpha_4}-1+\frac12i\tau\alpha_4\fe^{-it\alpha_4}\right)\,dt\right|
\lesssim \tau^3\alpha_4^2.
$$
Therefore, by using the estimate above, we obtain
\begin{align*}
\left| \int_0^\tau  \mathcal F\left(\widetilde{\mathcal{A}}_n(f_1,f_2,f_3)\right)(t,\xi)\,dt\right|
\lesssim \tau^3 \int_{\xi=\xi_1+\xi_2+\xi_3} |\xi|^{-1}\alpha_4^2\>\widehat f_1(\xi_1)\widehat f_2(\xi_2)\widehat f_3(\xi_3)\,(d\xi_1)(d\xi_2).
\end{align*}
By the symmetry, we may assume that $|\xi_1|\ge |\xi_2|\ge |\xi_3|$, then $|\xi|\le 3|\xi_1|$. Thus by \eqref{formu-a4}, we have
$$
|\xi|^{-1}\alpha_4^2\lesssim |\xi_1|^3|\xi_2|^2+|\xi_1|^2|\xi_2|^2|\xi_3|^2.
$$
By inverse Fourier transform, we get
\begin{align*}
\left| \int_0^\tau  \mathcal F\left(\widetilde{\mathcal{A}}_n(f_1,f_2,f_3)\right)(t,\xi)\,dt\right|
\lesssim \tau^3 \mathcal F\left(|\nabla|^3f_1\cdot |\nabla|^2f_2\cdot f_3+|\nabla|^2f_1\cdot |\nabla|^2f_2\cdot |\nabla|^2f_3\right)(\xi).
\end{align*}
Therefore, by Plancherel's identity and Lemma \ref{lem:kato-Ponce} (ii), we obtain that for any $\gamma_0\ge 0$,
\begin{align*}
\left\|\int_0^\tau\widetilde{\mathcal{A}}_n(f_1,f_2,f_3)(t)\,dt\right\|_{H^{\gamma_0}}
=& \left\| \int_0^\tau \langle\xi\rangle^{\gamma_0} \mathcal F\left(\widetilde{\mathcal{A}}_n(f_1,f_2,f_3)\right)(t,\xi)\,dt\right\|_{L^2}\\
\lesssim &
\tau^3 \left[\left\||\nabla|^3f_1\cdot |\nabla|^2f_2\cdot f_3\right\|_{H^{\gamma_0}}+\left\||\nabla|^2f_1\cdot |\nabla|^2f_2\cdot |\nabla|^2f_3\right\|_{H^{\gamma_0}}\right]\\
\lesssim &
\tau^3 \left\|f_1\right\|_{H^{\gamma_0+3}}\left\|f_2\right\|_{H^{\gamma_0+3}}
\left\|f_3\right\|_{H^{\gamma_0+3}}.
\end{align*}
This finishes the proof of the lemma.
\end{proof}

\subsection{Local error}
For the local error $\mathcal{\widetilde{L}}^n$ (\ref{local  2nd}), we have the following estimate.
\begin{lemma}\label{lem:local-error-2or} Let $\gamma \ge 0$, then
$$
\left\|\mathcal{\widetilde{L}}^n\right\|_{H^\gamma}\le C\tau^3,\quad n=0,\ldots,\frac{T}{\tau}-1,
$$
 where the constant  $C$ depends only on $\|u\|_{L^\infty_tH^{\gamma+3}_x}$.
\end{lemma}
\begin{proof}
We split $\mathcal{\widetilde{L}}^n$ into the following three parts as
$$
\mathcal{\widetilde{L}}^n=\mathcal{L}^n_1+\mathcal{L}^n_2+\mathcal{\widetilde{L}}^n_3,
$$
where $\mathcal{L}^n_1, \mathcal{L}^n_2$ are defined in \eqref{L1},  \eqref{L2} respectively, and
\begin{align*}
\mathcal{\widetilde{L}}^n_3=
&\frac1{18}\int_0^\tau \widetilde{\mathcal{A}}_n\left(v(t_n)\right)(s) \,ds.
\end{align*}

For $\mathcal{L}^n_1$, we have
\begin{align*}
\mathcal{L}^n_1=
&\frac12 \int_0^\tau\fe^{(t_n+s)\partial_x^3}
 \partial_x\left[\fe^{-(t_n+s)\partial_x^3}\left(v(t_n)+\frac12F_n\left(v(t_n),s\right)\right)\right]^2ds\\
 &-\frac12\int_0^\tau\fe^{(t_n+s)\partial_x^3}
 \partial_x\left(\fe^{-(t_n+s)\partial_x^3}v(t_n+s)\right)^2ds\\
 =&\frac12 \int_0^\tau\fe^{(t_n+s)\partial_x^3}
 \partial_x\Bigg[\fe^{-(t_n+s)\partial_x^3}\left(v(t_n)+\frac12F_n\left(v(t_n),s\right)-v(t_n+s)\right)\\
 &\qquad\qquad\qquad\qquad \cdot \fe^{-(t_n+s)\partial_x^3}\left(v(t_n)+\frac12F_n\left(v(t_n),s\right)+v(t_n+s)\right)\Bigg]ds.
\end{align*}
Hence, by Lemma \ref{lem:kato-Ponce} (i), we have
\begin{align*}
&\left\|\mathcal{L}^n_1\right\|_{H^{\gamma}}\\
\lesssim&
\tau\sup\limits_{s\in [0,\tau]}\left(\left\|v(t_n)+\frac12 F_n\left(v(t_n),\tau\right)-v(t_{n}+s)\right\|_{H^{\gamma+1}}\left\|v(t_n)+\frac12 F_n\left(v(t_n),\tau\right)+v(t_{n}+s)\right\|_{H^{\gamma+1}}\right).
\end{align*}
Similarly as the treatment in the proof of Lemma \ref{lem:local-error}, we get that for any $s\le \tau$,
\begin{align*}
\left\|v(t_n+s)-v(t_n)-\frac12 F_n\left(v(t_n),\tau\right)\right\|_{H^{\gamma+1}}
\lesssim& \tau^2 \|v\|_{L^\infty_t H^{\gamma+3}_x}^3,\\
\left\|v(t_n+s)+v(t_n)+\frac12 F_n\left(v(t_n),\tau\right)\right\|_{H^{\gamma+1}}
\lesssim & \|v\|_{L^\infty_t H^{\gamma+2}_x}+\|v\|_{L^\infty_t H^{\gamma+2}_x}^2.
\end{align*}
Hence, we obtain that
\begin{align}
\left\|\mathcal{L}^n_1\right\|_{H^{\gamma}}
\lesssim \tau^3\left( \left\|v\right\|_{L^\infty_t H^{\gamma+3}_x}^4+\left\|v\right\|_{L^\infty_t H^{\gamma+3}_x}^5\right).\label{est:Ln-1-r}
\end{align}

For $\mathcal{L}^n_2$, we recall that
\begin{align*}
\mathcal{L}^n_2=
&-\frac12\int_0^\tau\fe^{(t_n+s)\partial_x^3}
 \partial_x\left(\fe^{-(t_n+s)\partial_x^3}F_n\left(v(t_n),s\right)\right)^2ds.
\end{align*}
Hence by Lemma \ref{lem:kato-Ponce}, we have
\begin{align*}
\left\|\mathcal{L}^n_2\right\|_{H^{\gamma}}
\lesssim
\tau \sup\limits_{s\in [0,\tau]}\left\|F_n\left(v(t_n),s\right)\right\|_{H^{\gamma+1}}^2.
\end{align*}
Using \eqref{est:Fn}, we get
\begin{align}
\left\|\mathcal{L}^n_2\right\|_{H^{\gamma}}
\lesssim \tau^3 \|v\|_{L^\infty_tH^{\gamma+2}_x}^4.\label{est:Ln-2-r}
\end{align}

For $\mathcal{\widetilde{L}}_3^n$, by Lemma \ref{lem:An-2} (ii), we have
\begin{align}
\left\|\mathcal{\widetilde{L}}^n_3\right\|_{H^{\gamma}}
\lesssim \tau^3 \|v\|_{L^\infty_tH^{\gamma+3}_x}^3.\label{est:Ln-3-r}
\end{align}

Together with \eqref{est:Ln-1-r},  \eqref{est:Ln-2-r}  and \eqref{est:Ln-3-r}, the lemma is proved.
\end{proof}

\subsection{Stability}
The main result in this subsection is devoted for the propagator $\widetilde{\Phi}^n$ (\ref{stable 2nd}). To this purpose, we need the following $H^{\gamma+2}$-estimate on $v^n$.
\begin{lemma}\label{lem:apri} Let $\gamma \ge 0$, then for $n=0,\ldots,\frac{T}{\tau}-1$,
$$
\left\|{ v}^n\right\|_{H^{\gamma+2}}\le C,
$$
 where the constant  $C$ depends only on $\|u\|_{L^\infty_tH^{\gamma+3}_x}$.
 \end{lemma}
 \begin{proof}
 By using the same manner in the proof of Theorem \ref{thm:convergence} (the details are omitted here), we have that
 $$
 \left\|{ v}^n-v(t_n)\right\|_{H^{\gamma+2}}\le C\tau.
 $$
 Then the estimate is followed from the boundedness of $\left\|v(t_n)\right\|_{H^{\gamma+2}}$.
 \end{proof}
 Now we state our main result in this subsection.
\begin{lemma}\label{lem:stability-2or} Let $\gamma \ge 0$, then for $n=0,\ldots,\frac{T}{\tau}-1$,
$$
\left\|\widetilde{\Phi}^n({ v}^n)-\widetilde{\Phi}^n(v(t_n))\right\|_{H^\gamma}\le (1+ C\tau)\left\|{ v}^n-v(t_n)\right\|_{H^\gamma},
$$
 where the constant  $C$ depends only on $\|u\|_{L^\infty_tH^{\gamma+1}_x}$.
\end{lemma}
\begin{proof}
Note that
\begin{align*}
 \int_0^\tau\widetilde{\mathcal{A}}_n({ v}^n)(s)\,ds
  =&\int_0^\tau\mathcal{A}_n({ v}^n)(s)\,ds+\frac\tau2\fe^{t_n\partial_x^3}\partial_x^{-1}\left(\fe^{-t_n\partial_x^3}{ v}^n\right)^3
  -\frac{\tau}{2}\fe^{t_{n+1}\partial_x^3}\partial_x^{-1}\left(\fe^{-t_{n+1}\partial_x^3}{ v}^n\right)^3,
 \end{align*}
Then using the notations defined in the proof of Lemma \ref{lem:stability}, we write
\begin{align*}
&\widetilde{\Phi}^n({ v}^n)-\widetilde{\Phi}^n(v(t_n))
 =e_n+\Phi^n_1+\Phi^n_2+\Phi^n_3+\widetilde{\Phi}^n_0,
\end{align*}
where
\begin{align*}
\widetilde{\Phi}^n_0=&\frac\tau2\fe^{t_n\partial_x^3}\partial_x^{-1}\left[\left(\fe^{-t_n\partial_x^3}{ v}^n\right)^3
  -\left(\fe^{-t_n\partial_x^3}v(t_n)\right)^3\right]\\
&-\frac{\tau}{2}\fe^{t_{n+1}\partial_x^3}\partial_x^{-1}\left[\left(\fe^{-t_{n+1}\partial_x^3}{ v}^n\right)^3
  -\left(\fe^{-t_{n+1}\partial_x^3}v(t_n)\right)^3\right].
\end{align*}
Then it follows that
\begin{align*}
\left\|\Phi\left(v(t_n)\right)-\Phi\left(v^n\right)\right\|_{H^\gamma}^2
\le& \|e_n\|_{H^\gamma}^2+2\left\langle J^\gamma\Phi^n_1,J^\gamma e_n\right\rangle
+2\|e_n\|_{H^\gamma}\left\|\Phi^n_2\right\|_{H^\gamma}+2\|e_n\|_{H^\gamma}\left\|\Phi^n_3\right\|_{H^\gamma}
\\
&
+2\|e_n\|_{H^\gamma}\left\|\widetilde{\Phi}^n_0\right\|_{H^\gamma}+3\left\|\Phi^n_1\right\|_{H^\gamma}^2+3\left\|\Phi^n_2\right\|_{H^\gamma}^2
+3\left\|\Phi^n_3\right\|_{H^\gamma}^2+3\left\|\widetilde{\Phi}^n_0\right\|_{H^\gamma}^2.
\end{align*}
For $\left\langle J^\gamma\Phi^n_1,J^\gamma e_n\right\rangle$, since
\begin{align*}
\Phi^n_1=& \int_0^\tau
 \fe^{(t_n+s)\partial_x^3}\partial_x\Big(\fe^{-(t_n+s)\partial_x^3}
 e_n\cdot\fe^{-(t_n+s)\partial_x^3}\left(v^n+v(t_n)\right)\Big)ds,
 \end{align*}
 we have that
\begin{align*}
\left\langle J^\gamma\Phi^n_1,J^\gamma e_n\right\rangle
=& \int_0^\tau
 \left\langle J^\gamma \partial_x\left(\fe^{-(t_n+s)\partial_x^3}
 e_n\cdot\fe^{-(t_n+s)\partial_x^3}\left(v^n+v(t_n)\right)\right),J^\gamma \fe^{-(t_n+s)\partial_x^3}e_n\right\rangle ds.\end{align*}
Therefore, by Lemma \ref{lm3.5} (ii) and Lemma \ref{lem:apri}, it infers that
\begin{align}
\big|\left\langle J^\gamma\Phi^n_1,J^\gamma e_n\right\rangle\big|
\lesssim \tau\left\|e_n\right\|_{H^\gamma}^2\left(\left\|v(t_n)\right\|_{H^{\gamma+2}}+\left\|v^n\right\|_{H^{\gamma+2}}\right)\le C\tau\left\|e_n\right\|_{H^\gamma}^2.\label{est:phi-en}
\end{align}
Modifying the estimates similarly in the proof of Lemma \ref{lem:stability}, we also have
\begin{align}
\left\|\Phi_1^n\right\|_{H^\gamma}\le C
\sqrt\tau \left\|e_n\right\|_{H^{\gamma}};\quad
\left\|\Phi_2^n\right\|_{H^\gamma}+\left\|\Phi_2^n\right\|_{H^\gamma}\le C
\tau \left\|e_n\right\|_{H^{\gamma}}.\label{est:Phi-123}
\end{align}

Furthermore, following from Lemma \ref{lem:kato-Ponce},
we have that for any $\gamma_1>\frac12$, $f_j\in H^{\gamma}\cap H^{\gamma_1}$ for $j=1,2,3$ and any $t\geq0$,
 $$
 \left\|\fe^{t\partial_x^3}\partial_x^{-1}\left(\fe^{-t\partial_x^3}
 f_1\cdots\fe^{-t\partial_x^3}f_3\right)\right\|_{H^{\gamma}}
 \lesssim  \left\|f_1\right\|_{H^{\gamma}}\left\|f_2\right\|_{H^{\gamma_1}}\left\|f_3\right\|_{H^{\gamma_1}}.
 $$
Hence, using the above estimate and Lemma \ref{lem:apri}, we get
 \begin{align}\label{13.29}
\left\|\widetilde {\Phi}^n_0\right\|_{H^\gamma}\lesssim \tau\left\|e_n\right\|_{H^\gamma}\left(\left\|v(t_n)\right\|_{H^{\gamma_1}}^2+\left\|v^n\right\|_{H^{\gamma_1}}^2\right)\le C\tau\left\|e_n\right\|_{H^\gamma}.
 \end{align}
 Together with \eqref{est:phi-en}, \eqref{est:Phi-123} and \eqref{13.29}, we obtain that
\begin{align*}
\left\|\widetilde{\Phi}^n\left(v(t_n)\right)-\widetilde{\Phi}^n\left({ v}^n\right)\right\|_{H^\gamma}
\le& (1+ C\tau)\left\|{v}^n-v(t_n)\right\|_{H^\gamma}.
\end{align*}
This implies the desired result.
\end{proof}

\subsection{Proof of Theorem \ref{thm:convergence-2ord}}
Similarly as the proof of Theorem \ref{thm:convergence}, now by applying Lemma \ref{lem:local-error-2or} and Lemma \ref{lem:stability-2or} instead, we have Theorem \ref{thm:convergence-2ord}. \qed

\section{Numerical results}\label{sec:numerical}
In this section, we will carry out numerical experiments of the presented ELRI schemes
for justifying the convergence theorems. As benchmark for comparisons, we will show the results of the existing low-regularity integrators (LRIs) from \cite{kdv-kath},  and we shall denote the first order scheme therein as LRI1 and the second order scheme as LRI2 for short in the following. It is known from \cite{kdv-kath,WuZhao-1} that LRI1 can offer the first order accuracy in Sobolev space with two additional bounded spatial derivatives of the solution, and LRI2 can offer the second order accuracy with four additional bounded spatial derivatives.

To get an initial data with the desired regularity, we construct $u_0(x)$ for the KdV equation (\ref{model}) by the following strategy \cite{lownls}.
 Choose $N>0$ as an even integer and discretize the spatial domain $\bT=(0,2\pi)$ with grid points
$x_j=j\frac{2\pi}{N}$ for $j=0,\ldots,N$.
Take a uniformly distributed random vectors $\mathrm{rand}(N,1)\in[0,1]^N$ and denote
$$\mathcal{U}^N=\mathrm{rand}(N,1).$$
Then we define
\begin{equation}\label{non-smooth}
u_0(x):=\frac{|\partial_{x,N}|^{-\theta}\mathcal{U}^N}
{\||\partial_{x,N}|^{-\theta}\mathcal{U}^N\|_{L^\infty}},\quad
x\in\bT,
\end{equation}
where the pseudo-differential operator $|\partial_{x,N}|^{-\theta}$ for $\theta\geq0$ reads: for Fourier modes $l=-N/2,\ldots$, $N/2-1$,
\begin{equation*}
 \left(|\partial_{x,N}|^{-\theta}\right)
 _l=\left\{\begin{split}
 &|l|^{-\theta},\quad \mbox{if}\ l\neq0,\\
  &0,\qquad\ \, \mbox{if}\ l=0.
  \end{split}\right.
\end{equation*}
Thus, we get $u_0\in H^\theta(\bT)$ for any $\theta\geq0$.

We implement the spatial discretizations of the numerical methods within discussions by the Fourier pseudo-spectral method \cite{Shen}
with a large number of grid points $N=2^{14}$ in the torus domain $\bT$. We shall present the error
$u(x,t_n)-u^n(x)$ under $H^1$ or $L^2$ norm at the final time $t_n=T=1$, where the `exact' solution is obtained numerically by the ELRI2 scheme (\ref{def:LRI-2ord scheme})
with $\tau=10^{-4}$. Firstly, we perform convergence tests to verify our theoretical results. Figure \ref{fig:LRI} shows the error in $H^1$-norm of the first order methods ELRI1 (\ref{def:LRI-1ord scheme}) and the LRI1 scheme from \cite{kdv-kath},
by using different time step $\tau$ under the initial data $u_0$ in $H^3$ or $H^2$ as defined in (\ref{non-smooth}).
In Figure \ref{fig:LRI2}, we show the corresponding error in $L^2$-norm  of the second order methods ELRI2 (\ref{def:LRI-2ord scheme}) and the LRI2 scheme from \cite{kdv-kath,WuZhao-1}, under $u_0$ in $H^4$ or $H^3$.
Furthermore, to illustrate the efficiency of the proposed ELRIs over the existing LRIs, we show in Figure \ref{fig:cpu} the error of the ELRIs and LRIs against the computational time. The schemes are programmed here sequentially in MATLAB and the tests were run on a Mac with 3.6GHz Intel Core i7.
At last, to test the technique condition $\gamma>\frac12$ in Theorem \ref{thm:convergence} for ELRI1, in Figure \ref{fig:L2} we show the error of ELRI1 in $L^2$-norm under initial data from $H^1$ space.

\begin{figure}[t!]
$$\begin{array}{cc}
\psfig{figure=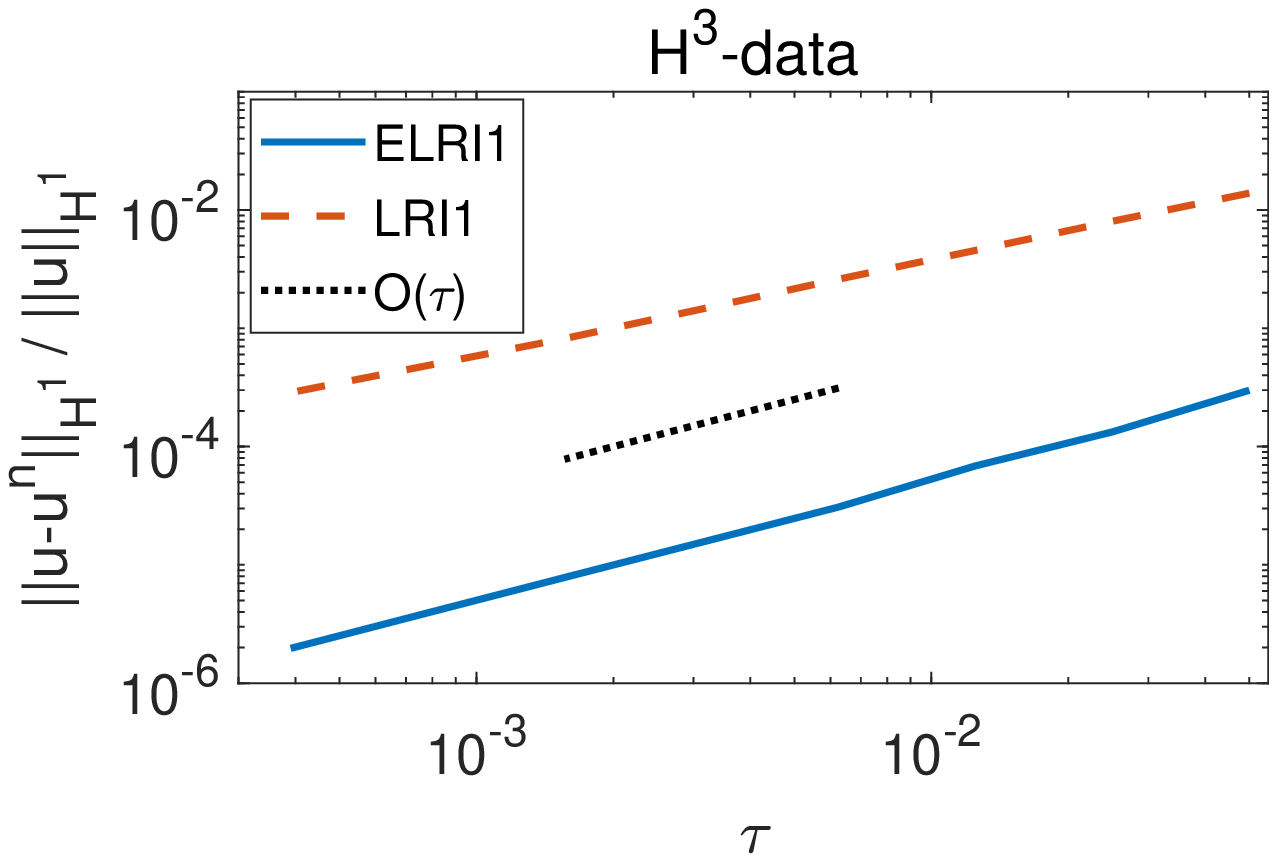,height=6cm,width=7cm}&
\psfig{figure=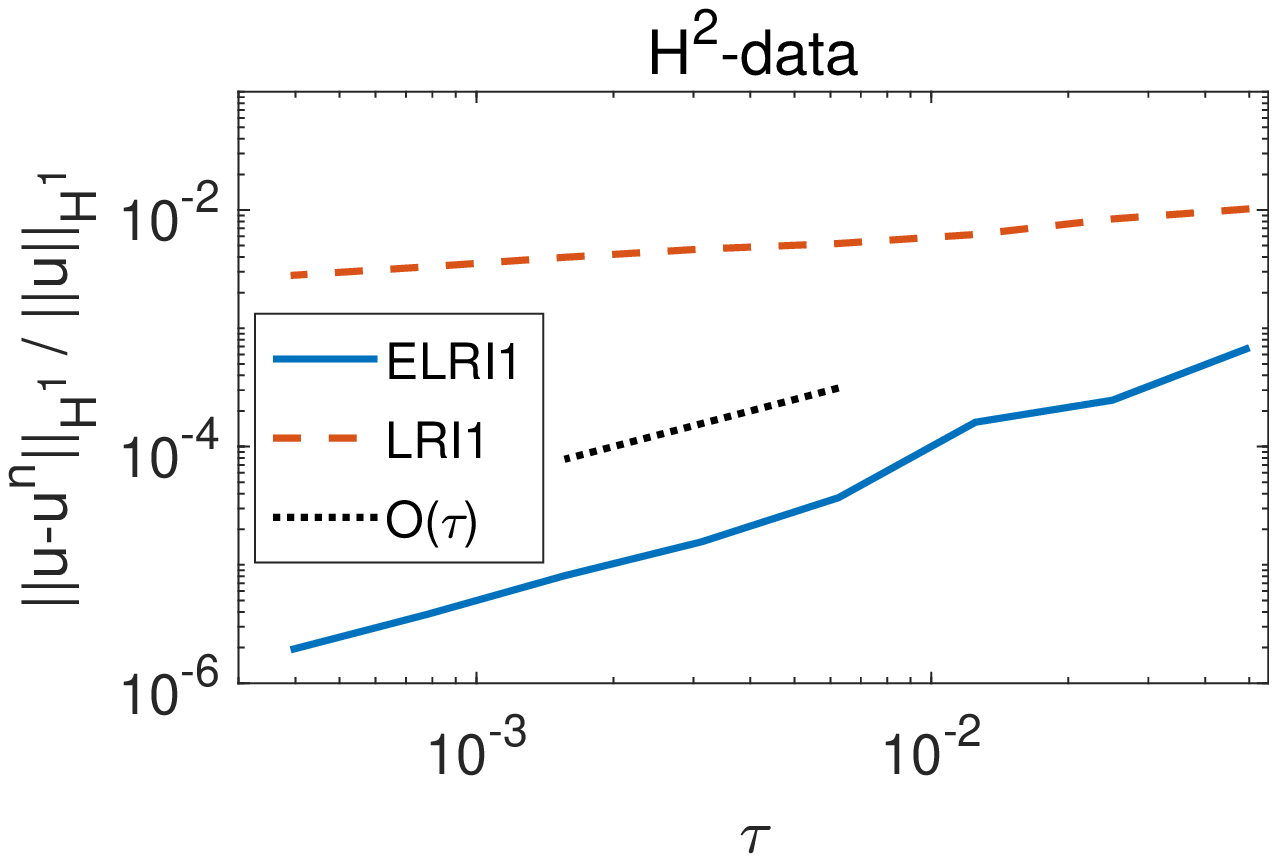,height=6cm,width=7cm}
\end{array}$$
\caption{Convergence of the first order methods:  relative error $\|u-u^n\|_{H^1}/\|u\|_{H^1}$ of ELRI1 and LRI1 at $t_n=T=1$ under
$H^3$-initial data (left) and $H^2$-initial data (right).
}
\label{fig:LRI}
\end{figure}

\begin{figure}[t!]
$$\begin{array}{cc}
\psfig{figure=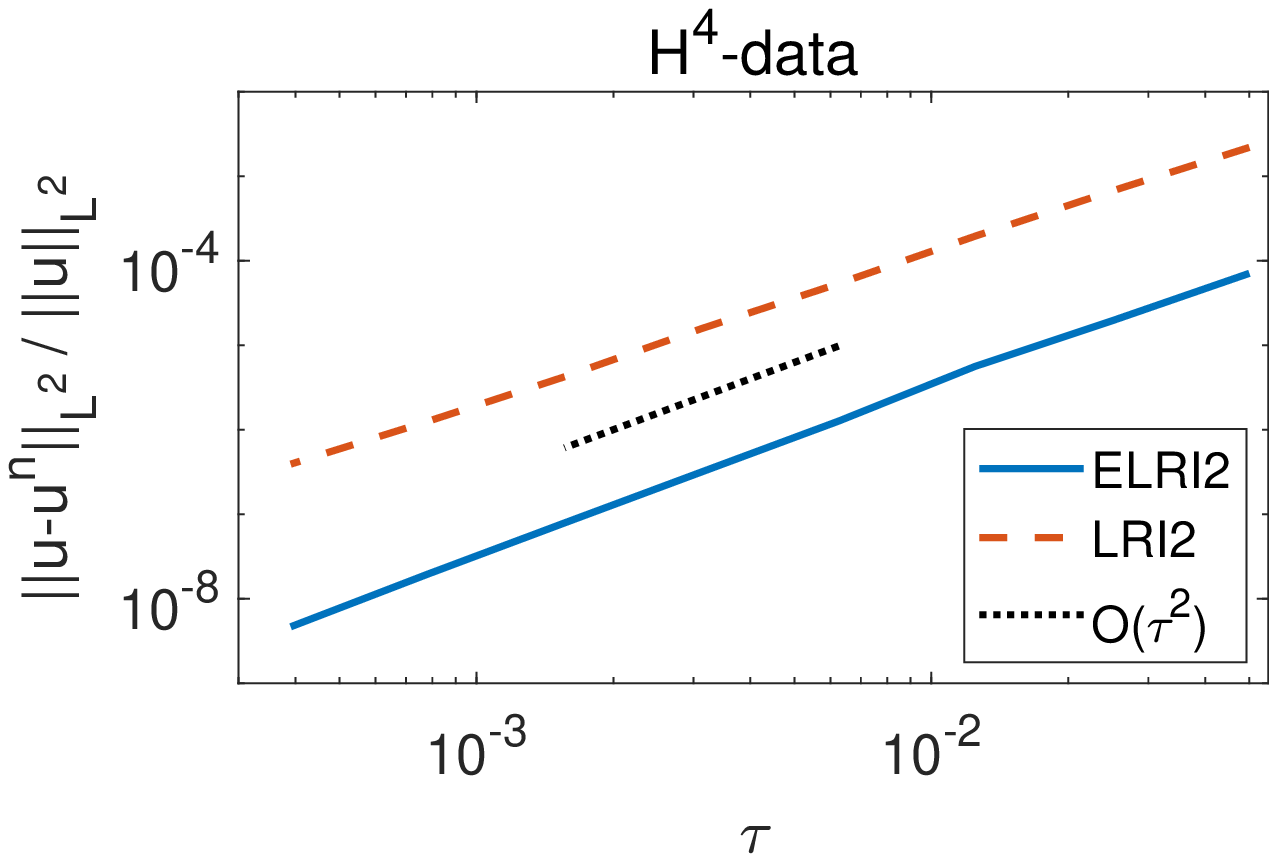,height=6cm,width=7cm}&
\psfig{figure=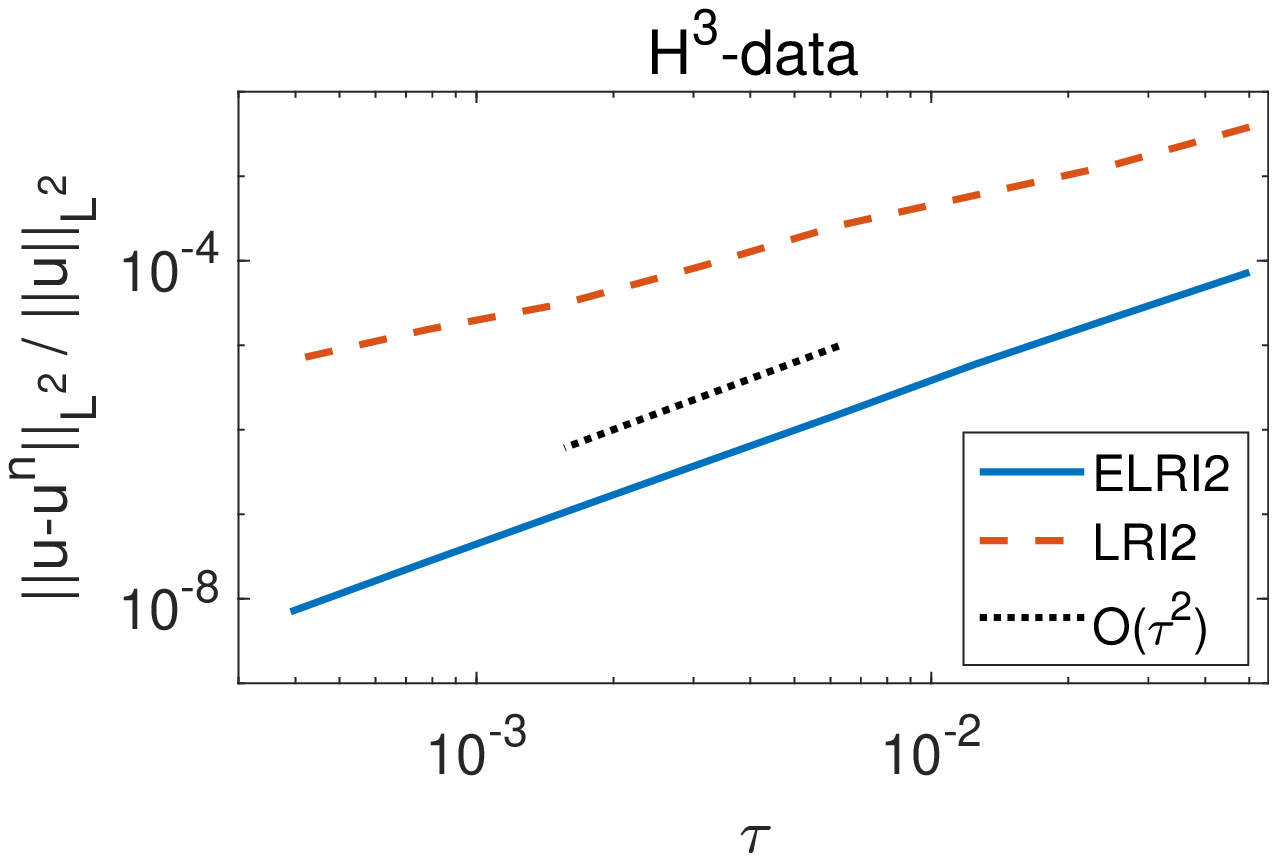,height=6cm,width=7cm}
\end{array}$$
\caption{Convergence of the second order methods:  relative error $\|u-u^n\|_{L^2}/\|u\|_{L^2}$ of ELRI2 and LRI2 at $t_n=T=1$ under
$H^4$-initial data (left) and $H^3$-initial data (right).
}
\label{fig:LRI2}
\end{figure}

\begin{figure}[t!]
$$\begin{array}{cc}
\psfig{figure=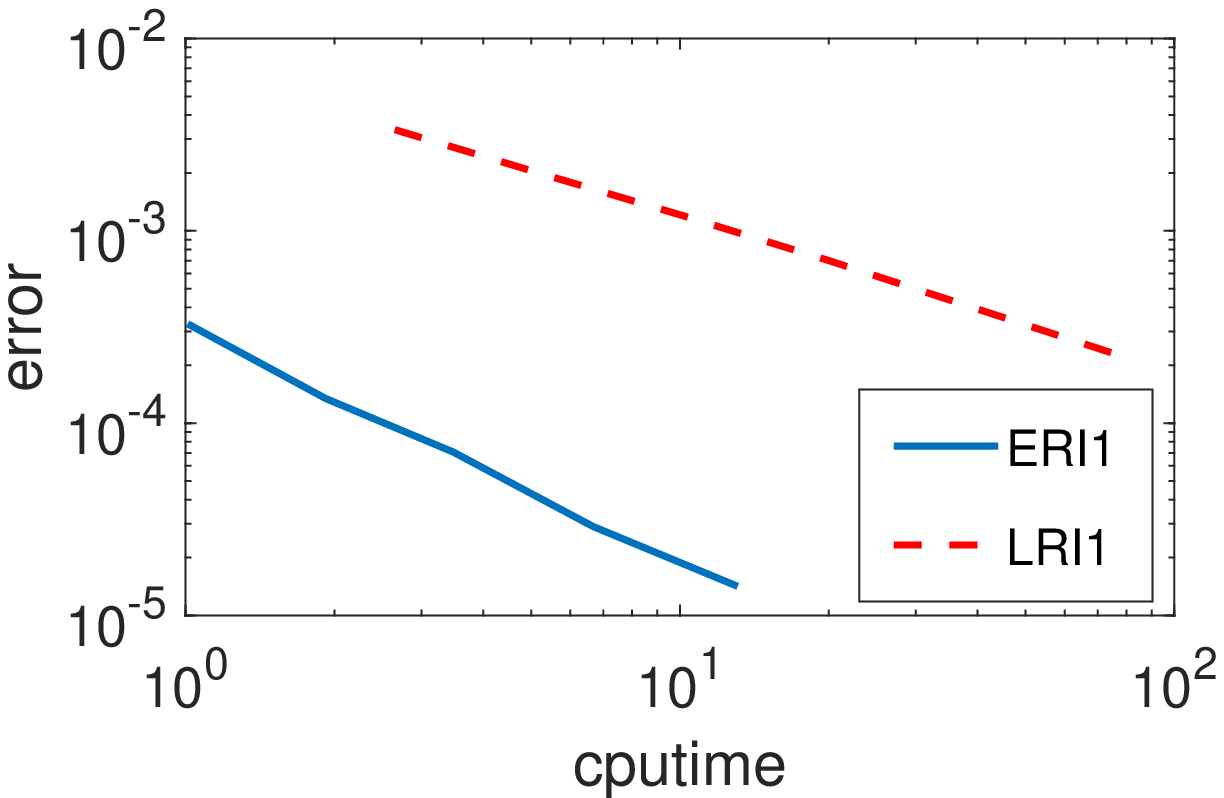,height=6cm,width=7cm}&
\psfig{figure=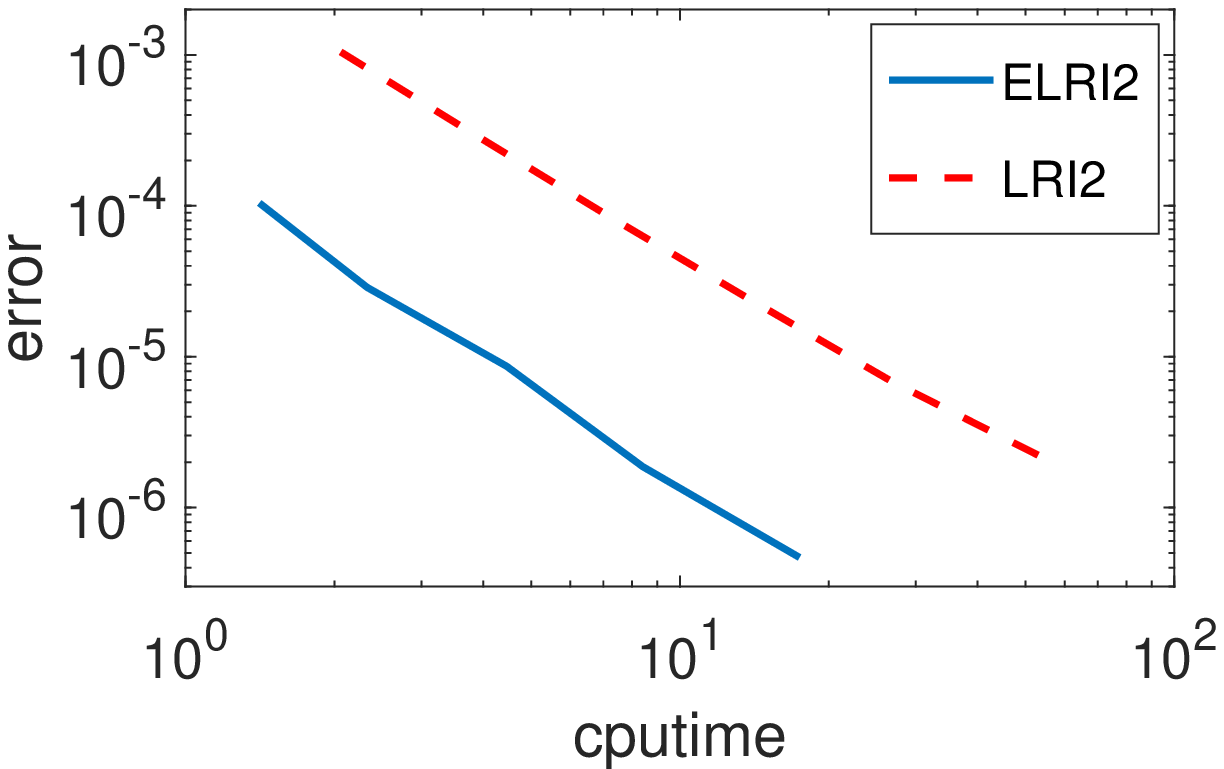,height=6cm,width=7cm}
\end{array}$$
\caption{Efficiency comparison between ELRIs and LRIs: the error $\|u-u^n\|_{H^1}/\|u\|_{H^1}$ of ELRI1 and LRI1 under $H^3$-data (left), and $\|u-u^n\|_{L^2}/\|u\|_{L^2}$ of ELRI2 and LRI2 under $H^4$-data (right) against computational time (cputime).
}
\label{fig:cpu}
\end{figure}

\begin{figure}[t!]
$$\begin{array}{c}
\psfig{figure=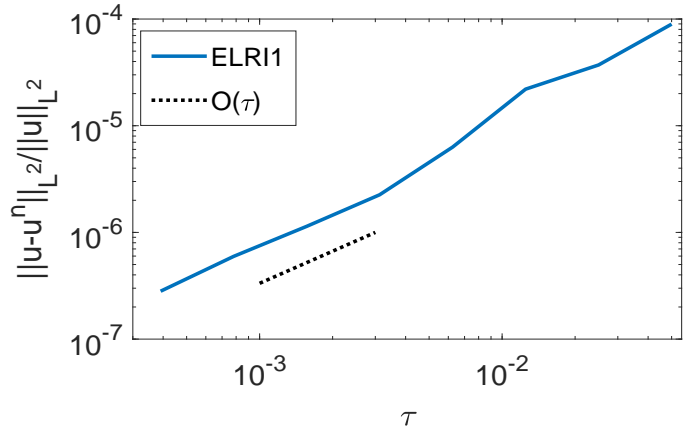,height=6cm,width=10cm}
\end{array}$$
\caption{Relative error $\|u-u^n\|_{L^2}/\|u\|_{L^2}$ of ELRI1 at $t_n=T=1$ under $H^1$-initial data.
}
\label{fig:L2}
\end{figure}

Based on the numerical results from Figures \ref{fig:LRI}-\ref{fig:L2}, we have the following clear observations:

1) The proposed ELRL1 scheme (\ref{def:LRI-1ord scheme}) shows the first order accuracy in $H^1$-norm with data in $H^{2}$ (see the right one of Figure \ref{fig:LRI}), and the ELRI2 (\ref{def:LRI-2ord scheme})
shows the second order accuracy in $L^2$-norm with data in $H^{3}$ (see the right one of Figure \ref{fig:LRI2}), which justify the theoretical results: Theorem \ref{thm:convergence} and Theorem \ref{thm:convergence-2ord}. In the contrast, the existing LRIs require more regularity to reach their optimal convergence rates.

2) The proposed ELRIs are much more accurate than the existing LRIs from \cite{kdv-kath} for solving the KdV equation (\ref{model}),  particularly when the data is less regular (see Figures \ref{fig:LRI}\&\ref{fig:LRI2}). To reach the same accuracy level, the LRIs cost much more computational time than the ELRIs (see Figure \ref{fig:cpu}).   Therefore, the ELRIs are more efficient for solving the KdV equation (\ref{model}) under rough data. For a comparison between the LRI method and the splitting method, we refer the readers to \cite{WuZhao-1}.

3) The technical condition $\gamma>\frac{1}{2}$ in the convergence theorem for ELRI1, i.e. Theorem \ref{thm:convergence} or Corollary  \ref{cor:ELRI}, is not needed in practical computing. Based on our numerical experience, we claim that  the optimal convergence rate of ELRI1 is true for any $\gamma\geq0$ (see Figure \ref{fig:L2}). While, the rigorous proof requires much more technical analysis and will be addressed in a future work.

\section{Conclusion} \label{sec:conclusion}
In this work, we have proposed two new embedded low-regularity integrators (ELRIs) for solving the KdV equation under rough data.
The new schemes are based on the exponential-type integration and a new embedded form together with some ideas borrowed from harmonic analysis. The rigorous convergence
theorems of the proposed ELRIs were established. It was shown that the ELRI schemes for the KdV equation could reach the first
order accuracy in $H^\gamma$-norm with initial data in $H^{\gamma+1}$ for $\gamma>\frac12$ and could reach the second order accuracy in $H^\gamma$-norm with
initial data in $H^{\gamma+3}$ for $\gamma\ge 0$, where the regularity requirements are lower than existing methods so far.  Numerical results were reported to justify the theoretical results
and comparisons are made with existing low-regularity integrators to show the improvements.

\section*{Acknowledgements}
Y. Wu is partially supported by NSFC 11771325 and 11571118.  X. Zhao is partially supported by the Natural Science Foundation of Hubei Province No. 2019CFA007, the NSFC 11901440 and the starting research grant of Wuhan University.

\bibliographystyle{model1-num-names}

\end{document}